\documentclass[12pt]{article}
\usepackage{amssymb,amsmath,amsthm,graphicx,ucs,hyperref}
\usepackage[utf8x]{inputenc}
\usepackage[english]{babel}

\newtheorem{proposition}{Proposition}
\newtheorem{theorem}{Theorem}
\newtheorem{question}{Open Question}

\title{Telstar Balls Gone Wild}
\author{Thomas Fernique}
\date{thomas.fernique@ens-lyon.org}

\begin{document}
\maketitle

\begin{abstract}
We describe an artistic project involving the fabrication of each of the 1812 distinct footballs that can be obtained by rearranging the way the 32 panels of a classic Telstar ball are stitched together.
Each ball is one of a kind and individually numbered, and the entire series is sorted by decreasing roundness.
\end{abstract}

\section{The project}

What could be more ordinary than a football made of black pentagons, each surrounded by white hexagons?
And yet, have you ever wondered what would happen if, in a fit of gentle madness, someone decided to stitch its pieces together in a different order?
What if we freed the pentagons and hexagons from this strict geometry and abandoned them to joyful disorder?
This is exactly what this project does, exploring all combinations, from the roundest to the strangest, keeping the exact same panels (12 black regular pentagons and 20 white regular hexagons).
The three pieces depicted in Figure~\ref{fig:ballons_cirm} may speak louder than words.

\begin{figure}[hbt]
\centering
\includegraphics[width=\textwidth]{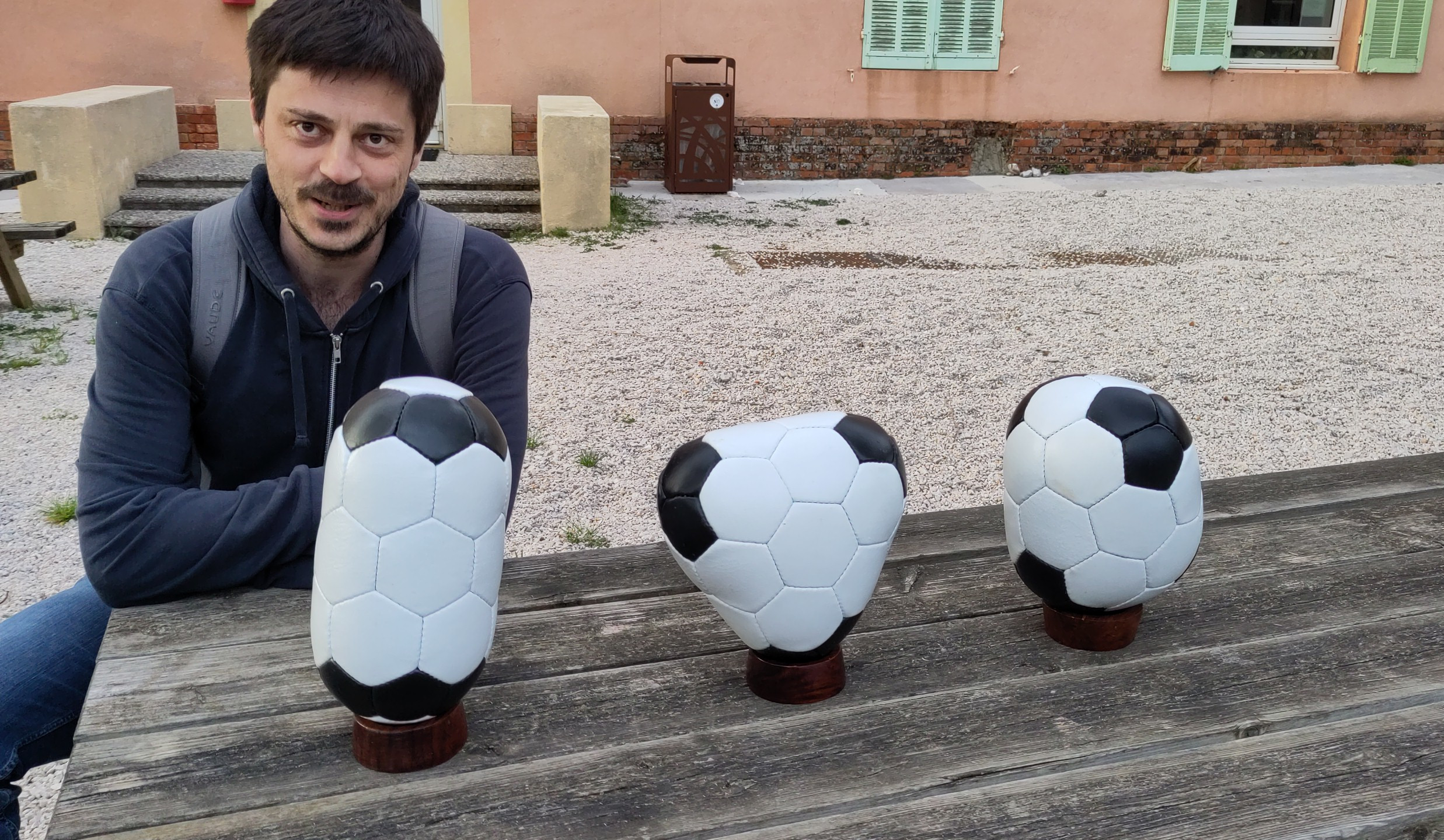}
\caption{Thomas Fernique and his three first balls (photo Sébastien Labbé).}
\label{fig:ballons_cirm}
\end{figure}

As we shall see, there are 1812 distinct balls up to isometry -- that is, up to rigid motion and mirror image.
The project raises both theoretical and practical questions: How are all the combinations to be enumerated? Are they genuinely feasible to construct from stitched leather panels? How might roundness be measured? How realistic is it to hand-stitch so many balls, each different and unusual? And what would an exhibition of such a multitude even look like?

Footballs have long inspired artists.
Most projects stick to the traditional round shape and focus on panel designs -- varying their shape or material.
Notable examples include the work of Jon-Paul Wheatley \cite{Wheatley}, balls exhibited at the OOF Gallery \cite{OOF}, and even the lineup of FIFA World Cup official match balls \cite{wikiballs}.
To the best of our knowledge, our project is original in its approach.
Perhaps the closest related work is Laurent Perbos's ``Ballon le plus long du monde'', which also uses the hexagons and pentagons of standard balls to craft cylindrical balls of varying lengths \cite{Perbos}.
Mention also the cubic ball of Fabrice Hybert \cite{Hyb98}, though the shape is not obtained from the natural curvature of the pieces (there are only hexagons).

Footballs also appear in science.
For example, to popularize Euler's formula for planar graphs -- see, in particular, Matt Parker's burlesque battle against impossible footballs \cite{bbc17, impossible}.
Or to illustrate mathematical notions such as the curvature of surfaces or their intrinsic geometry \cite{Ghy14,Ghy23}.
Or also in materials science through {\em fullerenes}, carbon molecules discovered in 1985 \cite{KARW85} whose bonds form hexagons or pentagons and can have various shapes \cite{SHSG17,SWA15,BBS25}.

Beyond the football inspiration, our approach belongs to a broader tradition of combinatorial or systematic art: generating all possible objects that satisfy a given set of rules.
This is a path explored by artists such as François Morellet \cite{Mor58}, Sol LeWitt \cite{LeW74} or Manfred Mohr \cite{Moh75}; it is still very much alive today \cite{Las15,Mai15,VJ26}.
The idea is to use mathematics as a way of revealing the unexpected richness hidden behind seemingly simple objects. This is a common thread in many works presented at the Bridges conference on ``Mathematics and the Arts'', where we presented this project in 2026 \cite{Fer26}.


\section{Combinatorics}
\label{sec:combinatorics}

Let us first focus on the combinatorics of the panels, that is, the way pentagons and hexagons can be assembled to form a topological sphere.
We here do not pay attention to their embedding in the 3-dimensional Euclidean space: we are looking at {\em abstract graphs}, which are just vertices connected by edges.
More precisely, we want planar graphs (that can be drawn without edge crossing on a plane or a sphere) that are 3-regular (because the panels always meet by 3) and with only hexagonal or pentagonal faces.
Such graphs exactly correspond to the structure of the fullerene molecules already mentioned and, by extension, we also call them fullerenes.

\subsection{Twelve pentagons}

It may be worth first recalling a standard application of the Euler relation:

\begin{proposition}
\label{prop:12p}
Every fullerene has exactly 12 pentagons.
\end{proposition}

\begin{proof}
Every planar graph with $V$ vertices, $E$ edges and $F$ faces satisfies the Euler relation $V-E+F=2$.
Here, $F=p+q$ where $p$ and $q$ denote the number of pentagons and hexagons.
Moreover, since the graph is 3-regular, every vertex sees 3 edges and every edges is seen by two vertices (its endpoints): this yields $3V=2E$ and the Euler relation becomes $3F-E=6$.
Similarly, every pentagon (resp. hexagon) sees 5 (resp 6) edges and every edges is seen by two faces, hence $5p+6q=2E$ and the Euler relation becomes $6F-5p-6q=12$.
With $F=p+q$ this yields $p=12$, as claimed.
\end{proof}

As it appears in the above proof, the Euler formula does not yield any restriction on the number of hexagons.
Actually, any number is possible, with the only exception of a single hexagon (this can be seen on the ball No.~1810 in \cite{appli}: it is made of a spiral of hexagons that can have arbitrary length $k\geq 2$; the case $k=0$ is the usual dodecahedron).
But in our case, we are interested in fullerenes with 20 hexagons, as the classic Telstar football.
Specifically, we want to know how many different such fullerenes exist.

\subsection{Algorithmic generation}

Fortunately enough, different algorithms have been developed to generate fullerenes.
The most efficient one today appears to be {\em buckygen} \cite{BGM12}.
The principle of the algorithm is to define elementary operations that insert hexagons in a given fullerene to get new (larger) fullerenes, so that starting with suitable small fullerenes (seeds) any other fullerene can be generated.
The main issue is that there are exponentially many different sequences of elementary operations that yields eventually the same fullerene, and one has somehow to define a sort of ``canonical way'' to generate a fullerene in order to not generate the same fullerene many times.
The fullerenes have been constructed and counted up to $400$ vertices \cite{BGM12}.
In particular, there are exactly $1812$ fullerenes with 20 hexagons.
For our project, we used the algorithm buckygen to generate all these graphs.

\subsection{Mathematical counting}

Let us also mention a recent impressive result: a closed formula for the generating series of fullerenes \cite{ESG25}, that is, a series $\sum a_qx^q$ where $a_q$ is the number of different fullerenes with $q$ hexagonal faces.
The formula is a linear combination of 45 Eisenstein series -- special functions from number theory that have deep ties to arithmetic.
Though complicated, the formula allows to go far beyond the algorithmic enumeration.
For example, it shows that there are $1\,203\,397\,779\,055\,806\,181\,762\,759$ fullerenes with $10\,000$ vertices.
However, this gives only the {\em number} of fullerenes, whereas for our project we need the fullerenes {\em themselves}.

\section{Geometry}
\label{sec:geometry}

We now have 1812 abstract graphs (previous section).
But we want footballs, which are surfaces made of solid leather.
Is it physically possible?
A few attempts with paper suggest that this is the case (Fig.~\ref{fig:paperfold}), but we would like to be sure that this is not due to a somewhat messy assembly or a stroke of luck on these particular balls.
Formally, the question can be stated as follows: is it possible to embed these abstract graphs in 3-dimensional Euclidean space so that the faces are regular pentagons and hexagons, possibly bent but not (too much) stretched?

\begin{figure}[hbt]
\centering
\includegraphics[width=\textwidth]{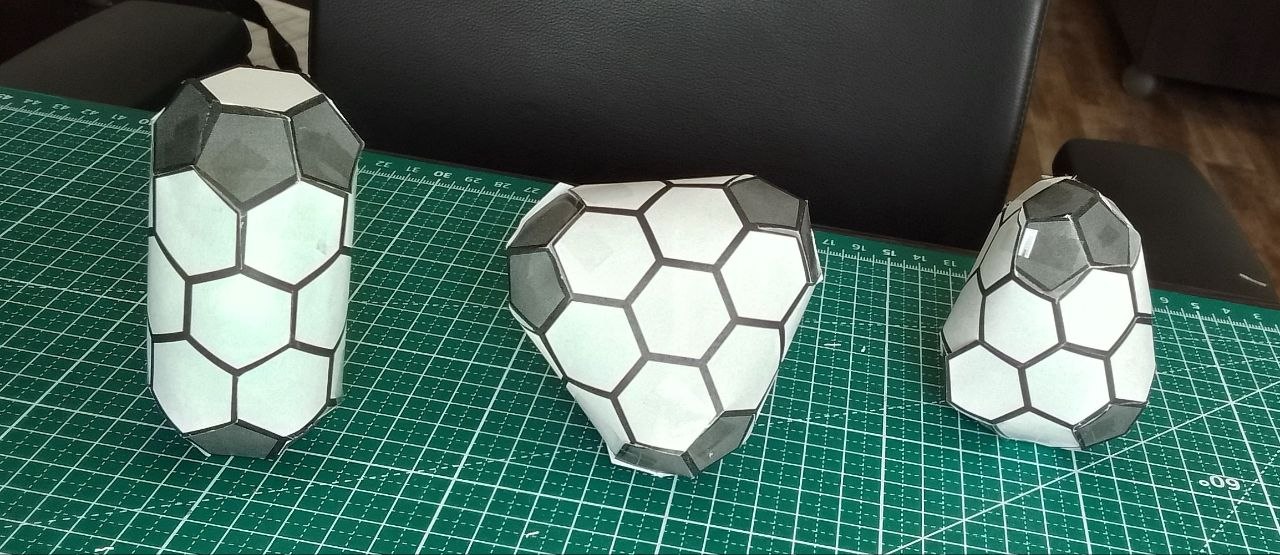}
\caption{
First trial with paper folding (before realizing the balls in Fig.~\ref{fig:ballons_cirm}).
}
\label{fig:paperfold}
\end{figure}

\subsection{Flat panels}

The issue is not so much related to the fact that regular pentagons and hexagons have unit-length edges.
Indeed, it is always possible to embed a fullerene so that all the edges have unit length (this follows from Theorem~1 in \cite{FKS20}).
The real issue lies in embedding the faces themselves.
Of course, it is tempting to embed faces as flat surfaces.
The Steinitz theorem \cite{SR34} indeed ensures that every fullerene (more generally, every planar 3-connected graph) can be obtained as the skeleton of a convex polyhedron (that is, the graph defined by the vertices and edges of the polyhedron).
However, there is no control on the shapes of the faces of this polyhedron: they may be not regular pentagons and hexagons.
Actually, it turns out to be possible only for the classic football:

\begin{proposition}
\label{prop:flat_faces}
The only fullerenes that can be embedded with flat regular faces are the dodecahedron and the truncated icosahedron.
\end{proposition}

\begin{proof}
Define the {\em type} of a vertex as the ordered tuple of the number of sides of the three faces that meet at this vertex, that is, $(5,5,5)$, $(5,5,6)$, $(5,6,6)$ or $(6,6,6)$.
The point is that the type of an endpoint of an edge defines the {\em dihedral angle} along this edge (the angle made by the two faces that share this edge).
The type of the other endpoint must hence be identical, in order to define the same angle.
By propagation, all vertices must have the same type (such a polyhedron is said to be {\em Archimedean}).

More precisely, the type $(5,5,5)$ yields the dodecahedron (12 pentagons and no hexagon), while the type $(5,6,6)$ yields the truncated icosahedron (the classic football).
The type $(6,6,6)$ is excluded because it would yield an infinite hexagonal grid.
The type $(5,5,6)$ is also excluded because every pentagon would be surrounded by alternating hexagons and pentagons, which is incompatible with an odd number of neighbors.
\end{proof}

\subsection{Bending the panels}

The panels are however not {\em rigid} but {\em solid}, that is, we can {\em bend} them.
Does that help?
As it happens, yes.
This follows from Alexandrov's theorem on polyhedra, stated in \cite{Ale05} as follows (existence is established in Chapter~4 and unicity in Chapter~3):

\begin{theorem}
\label{th:alexandrov}
A development homeomorphic to the sphere with the sum of angles at most $2\pi$ at each vertex defines a unique convex polyhedron by gluing.
\end{theorem}

Let us explain this statement and how it applies in our case.
A {\em development} is a collection of polygons with some prescribed rule for {\em gluing} them together along their sides, that is, instructions that specify which side of each polygon is glued to a side of another or the same polygon and in which direction.
It is {\em homeomorphic} to the sphere if one can continuously map the points of the polygons to the sphere, and conversely.
In our case, the polygons of the development are the 12 pentagons and the 20 hexagons.
The gluing is specified by the considered abstract graph (fullerene), and the planarity of this abstract graph ensures that the development is homeomorphic to the sphere.
Moreover, since at each vertex meet three polygons whose largest angle is $2\pi/3$ (for the hexagon), the sum of angles at each vertex is indeed at most $2\pi$. 
The Alexandrov theorem thus applies: every fullerene can be embedded as a (unique) convex polyhedron in 3-dimensional Euclidean space.
In other words: all our footballs do exist!

\begin{figure}[hbt]
\centering
\includegraphics[width=\textwidth]{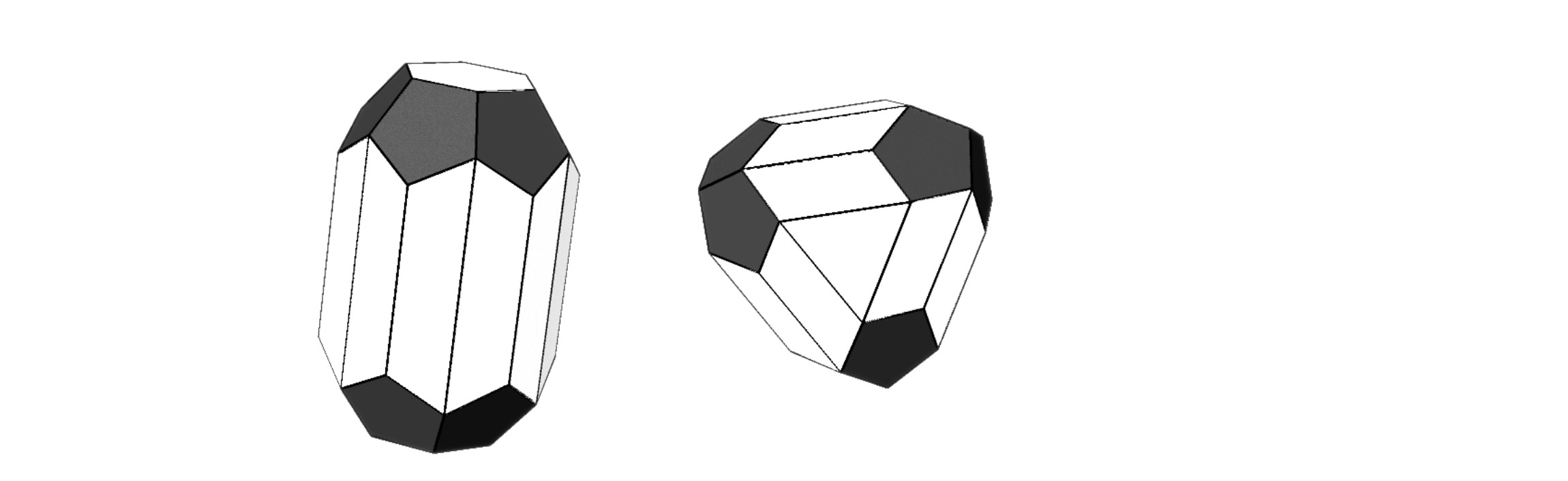}
\caption{
The Alexandrov polyhedra defined by the fullerenes corresponding to the footballs depicted in Fig.~\ref{fig:ballons_cirm}.
The first one has 26 faces, the second one 29 and the third one remains to be found!
The pentagonal faces are here all embedded as flat faces, but this is not always the case.
}
\label{fig:alexandrov}
\end{figure}

\subsection{Where are the creases?}

Let us stress that the edges of the Alexandrov polyhedron usually do not coincide with those of the pentagons and hexagons (except for the dodecahedron and the truncated icosahedron, Prop.~\ref{prop:flat_faces}).
Some examples are given in Figure~\ref{fig:alexandrov}.
Actually, it is usually not easy to find the Alexandrov polyhedron of a given fullerene.
The original proof of Theorem~\ref{th:alexandrov} is indeed not {\em constructive} (the set of developments that admit a polyhedral embedding is proven to be closed and open in the connected space of developments).
Another more recent proof provides a differential equation that progressively transforms an arbitrary initial polyhedron into the wanted Alexandrov polyhedron \cite{BI08}.
It has been shown to yield a pseudopolynomial algorithm \cite{KPD09}.
The theoretical bound on this algorithm -- $O(n^{457}/\varepsilon^{1182})$ for an $\varepsilon$-approximation of a polyhedron with $n$ vertices -- is however large (though it may work much better in practice) and we get the polyhedron only in the limit $\varepsilon\to 0$.
In particular, the following question remains open:

\begin{question}
Find the Alexandrov polyhedra defined by the fullerenes with $20$ hexagons.
\end{question}

\subsection{Avoiding creases?}

Though Alexandrov's theorem ensures the existence of our football, the embedding as a polyhedron, whose edges correspond to {\em creases} in the football, may be disappointing.
Could we not find a smoother embedding, e.g., {\em continuously differentiable}?
The {\em $C^1$-embedding theorem} by Nash and Kuiper ensures the existence of such an embedding (\cite{EM02}, Chapter~21), but it is not really a good fit for our context.
The surface, obtained at the limit by adding infinitely many corrugations with increasingly higher frequencies and smaller amplitudes, looks indeed quite pathological (the interested reader can look at the stunning pictures of such an embedding obtained in \cite{BJLT12} for the flat torus).
In particular, the embedding is not convex at all, and it was proven by Pogorelov that it cannot be so (\cite{Pog73}, Theorem~1 p.~167).
We can only hope that the slight elasticity of the panels on our balls will round off the creases and the corners.
After all, this is what happens with a classic football, which is closer to a sphere than to a truncated icosahedron with sharp edges.
We will see that this also appears to be the case for our balls.

\subsection{Chirality}

To end this section, let us say a word about {\em chirality}.
This term, widely used in molecular chemistry, derives from the Greek word for hand and alludes to the fact that although left and right hands look nearly identical, they are distinct: they are related by mirror symmetry but cannot be superimposed.
This is precisely the case for the majority of the 1812 fullerenes -- 1722 of them, to be exact: they admit two distinct embeddings that are mirror images of each other (Fig.~\ref{fig:mirrored}).
Thus, these 1812 fullerenes potentially yield $1812 + 1722 = 3534$ different footballs.
However, we have chosen -- at least for now -- to consider only one of the two possible mirror embeddings, and therefore to restrict ourselves to 1812 footballs.

\begin{figure}[hbt]
\centering
\reflectbox{\includegraphics[width=\textwidth]{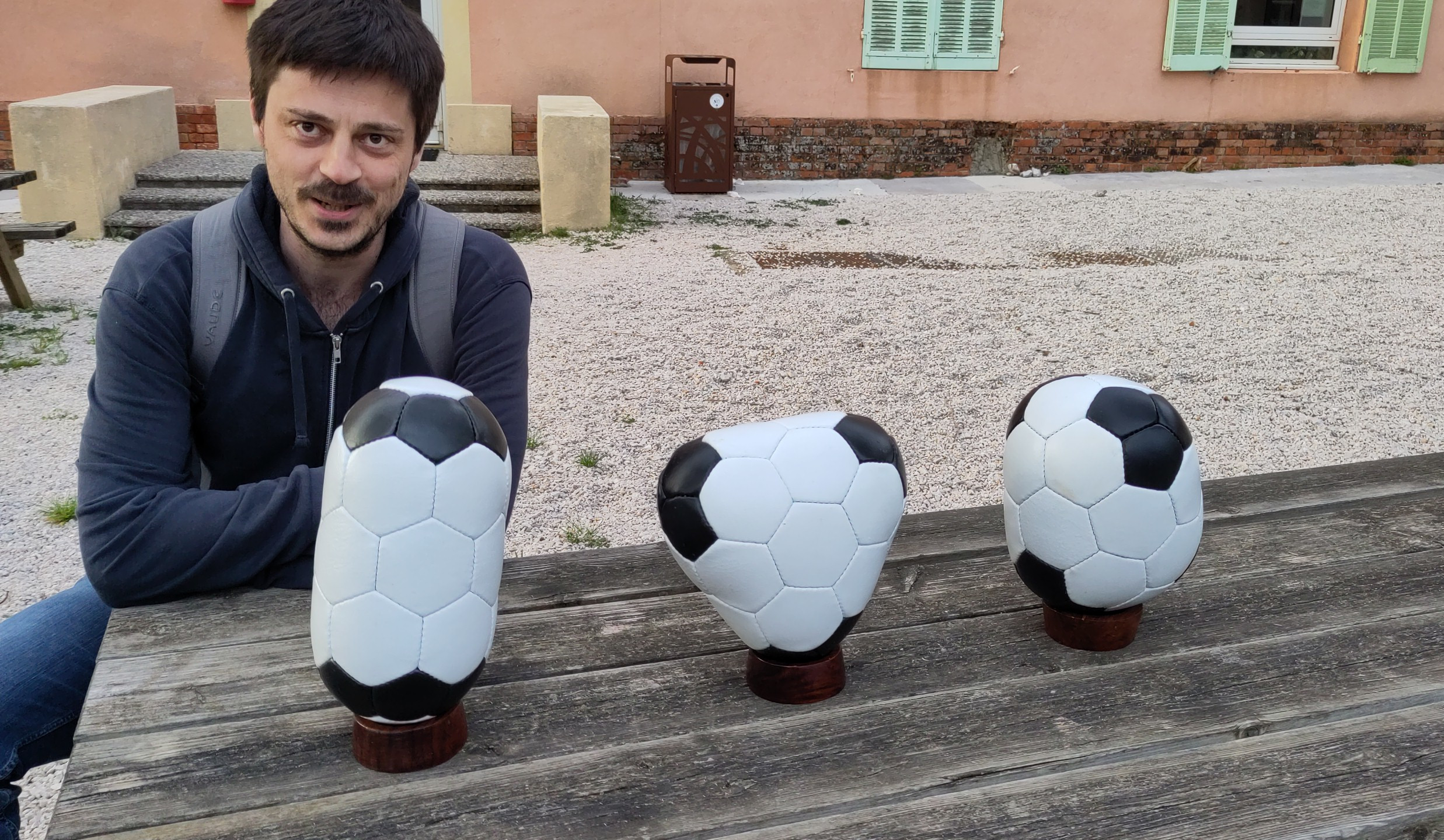}}
\caption{
Mirror image of Fig~\ref{fig:ballons_cirm}.
The football in front of the author remains the same because it has mirror symmetry, but the other two cannot be superimposed on those in Fig.~\ref{fig:ballons_cirm}: they are chiral (the author as well).
}
\label{fig:mirrored}
\end{figure}

\section{Modeling \& numbering}
\label{sec:modeling}

Now that we know our footballs exist (previous section), let us model them.
This will allow us to write a JavaScript applet for visualizing every ball, with free rotation and zoom, which proves very useful for selecting the next ball to make and verifying it afterwards \cite{appli}.
This will also allow us to estimate the volume of each football before fabrication, which is crucial because we decided to number the balls in decreasing order of volume -- and this unique number is stamped on the valve panel of each ball before the panels are stitched.
We chose this method over the admittedly more accurate one of making all 1812 balls first, immersing each in water to measure the displaced volume, then unstitching them to stamp the corresponding volume number on each before finally sewing them back together\ldots

\subsection{Vertices embedding}

To model a football, we drew inspiration from so-called force-driven graph drawing algorithms \cite{Ead84}.
We added ``virtual'' springs between pairs of vertices (neighboring or not) of each panel, as well as a repulsive force from the centroid of the vertices toward each vertex.
Then, starting from an arbitrary initial embedding (Fig.~\ref{fig:modeling1}, left), we applied Newton's first law of motion to continuously move the vertices until an equilibrium position was reached.
The springs are intended to model the slightly elastic behavior of leather; their stiffness is a parameter to be chosen carefully (in particular those that cross the panel: they control how easily the panel can be bent).
The repulsive force is intended to model the force exerted by the bladder on the surface, which is significant (usually around 11 PSI -- Pounds per Square Inch -- roughly the force needed to push a heavy fire door); we chose a force in $r^6$, consistent with the theory of rigid membranes.
This process gave us an embedding for the {\em vertices} of each football, but not really for the faces (Fig.~\ref{fig:modeling1}, right).

\begin{figure}[hbt]
\centering
\includegraphics[width=\textwidth]{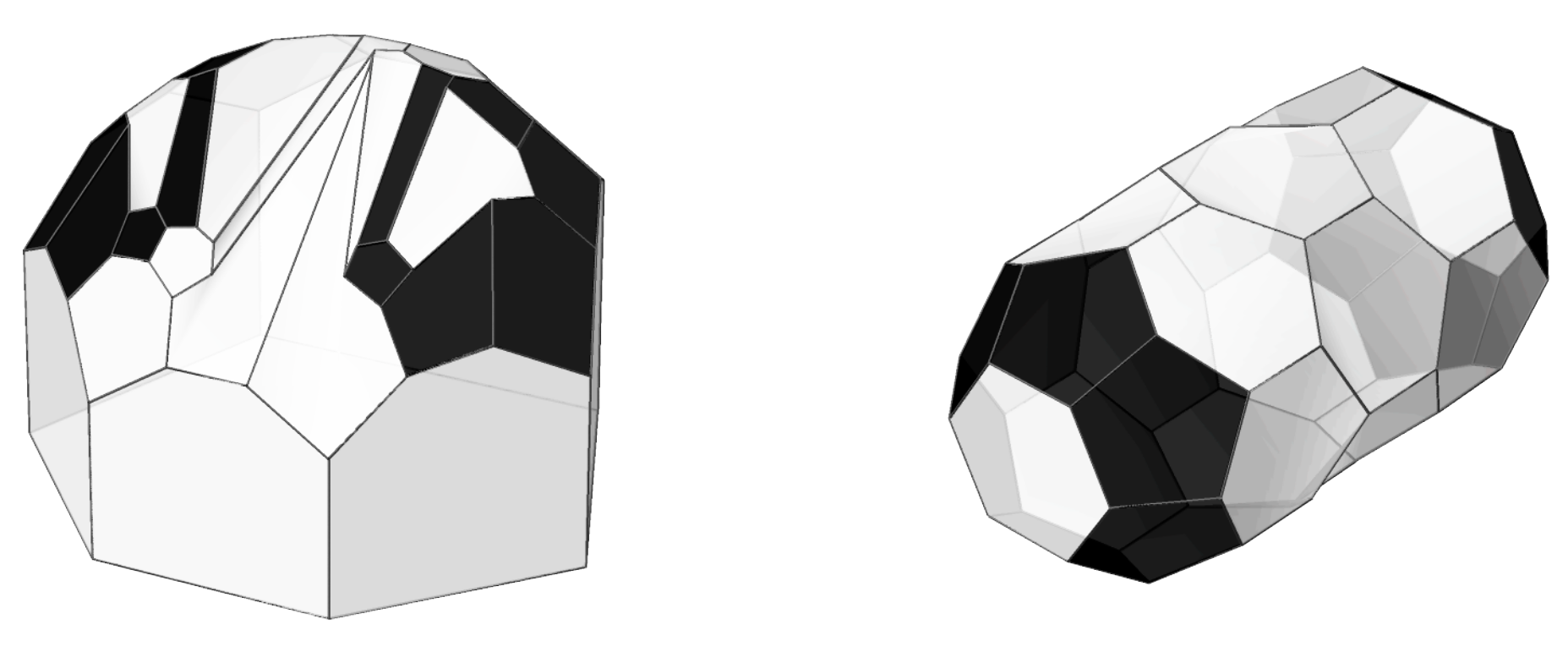}
\caption{
An arbitrary initial embedding of the vertices (left) and the embedding produced by our force-driven algorithm after 1000 iterations.}
\label{fig:modeling1}
\end{figure}

\subsection{Face embedding}

Now, to turn the skeleton defined by our vertices and edges into a surface, in other words to embody our footballs, we use a classic technique in computer graphics: {\em subdivision}.
Namely, we apply a few iterations of the classic Catmull--Clark algorithm \cite{CC78}; three iterations proved to be a satisfactory compromise (Fig.~\ref{fig:modeling2}).

\begin{figure}[hbt]
\centering
\includegraphics[width=\textwidth]{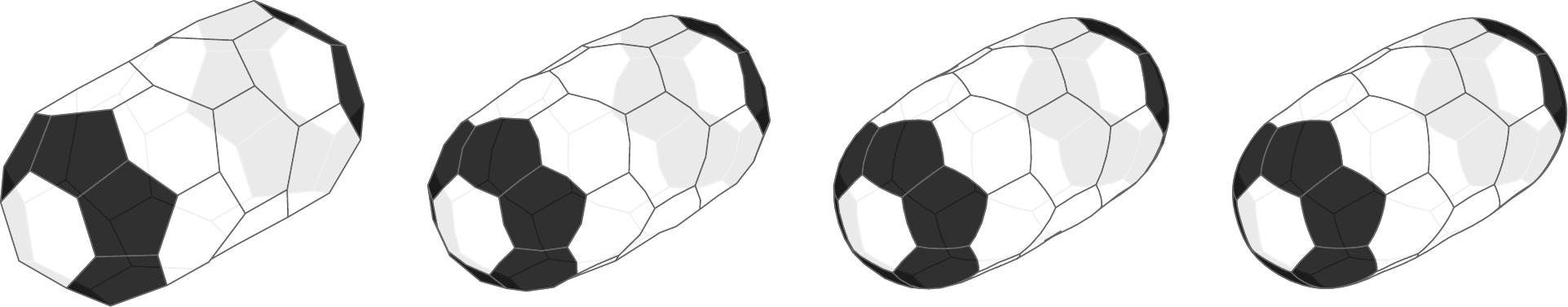}
\caption{
Successive applications of the Catmull–Clark algorithm (from left to right: 32, 180, 720 and 2880 faces).
}
\label{fig:modeling2}
\end{figure}


We do not claim to be performing a perfect simulation of the leather behavior, though the model seems to be quite faithful, and we were able to validate it on about 50 test balls -- see for example Fig.~\ref{fig:modeling3}.
The rapid generation, via Catmull–Clark, of the 2880 faces of a ball model from the initial 60-vertex embedding is also a significant advantage, particularly for the visualization applet \cite{appli}.

\begin{figure}[hbt]
\centering
\includegraphics[width=\textwidth]{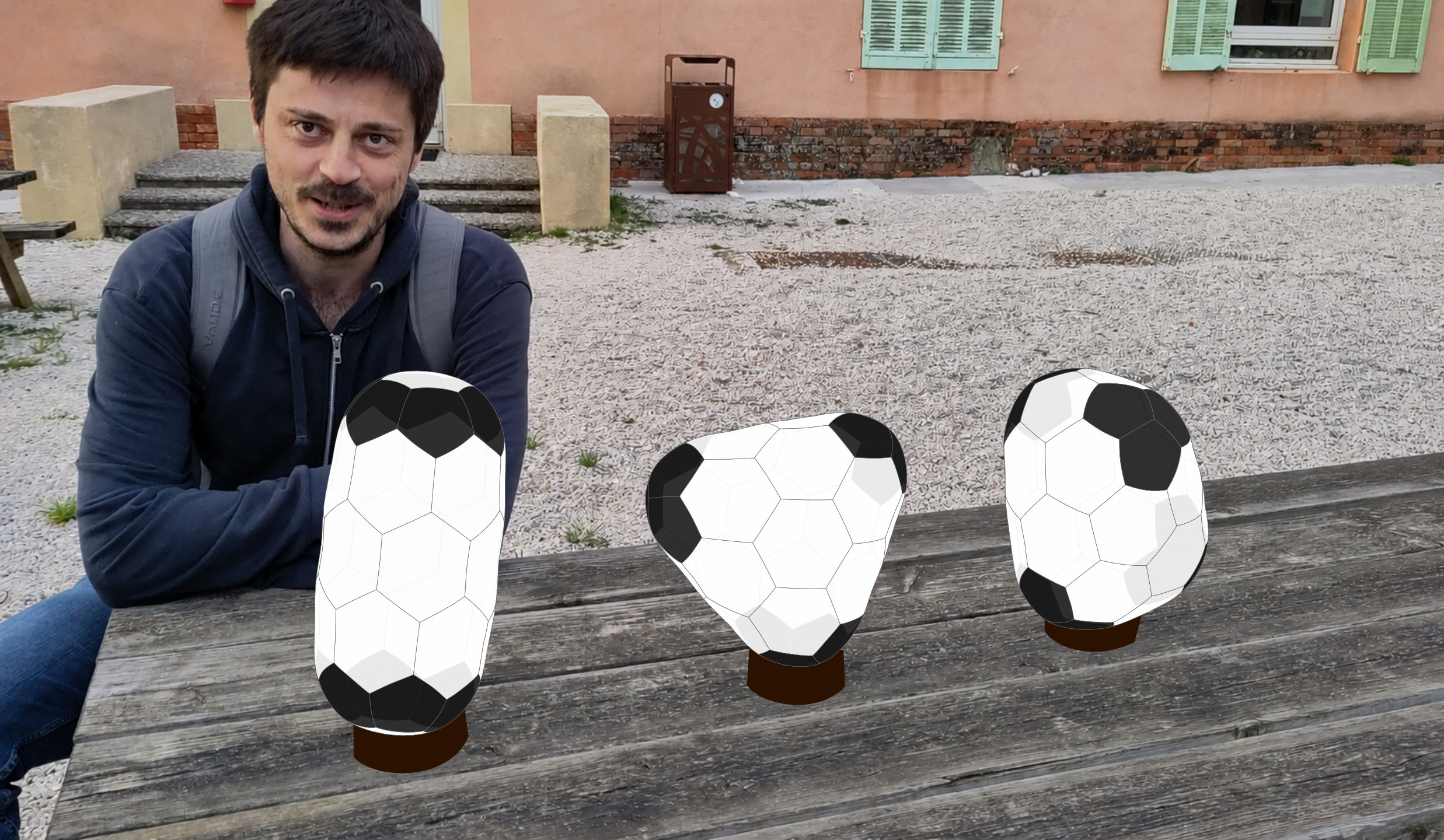}
\caption{Figure~\ref{fig:ballons_cirm} with the balls replaced by their models.}
\label{fig:modeling3}
\end{figure}

Let us also mention that our model has no reason to match the models elaborated by chemists for fullerenes, see, e.g., \cite{SHSG17}.
These aim to predict the position of atoms by molecular mechanics, which is completely different from modeling the shape of a surface resulting from leather mechanical constraints.
For example, even the Buckminsterfullerene, that is, the fullerene whose structure corresponds to the classic football, is not expected to have exactly the shape of a truncated icosahedron \cite{HMLLH91,FPS16}.

\subsection{Numbering}

As mentioned, we sorted the balls in decreasing order of their volume $V$, as computed from the model described above.
Actually, we refer to their {\em roundness}, defined as the {\em isoperimetric quotient} $V^2/S^3$ -- where $S$ is the surface area of the ball.
We normalize roundness at 100 for the roundest ball -- the classical one.
Since all the footballs have the same surface area $S$ and the function $V\mapsto V^2$ is increasing, sorting by roundness actually amounts to sorting by volumes.
For example, the balls depicted in Fig.~\ref{fig:ballons_cirm} have number 1796, 1710 and 1346 (from left to right) and roundness 66, 76 and 83.
Fig.~\ref{fig:roundness} depicts the roundness distribution, while Fig.~\ref{fig:toutes} provides an overview of all the 1812 balls.

\begin{figure}[hbt]
\centering
\includegraphics[width=\textwidth]{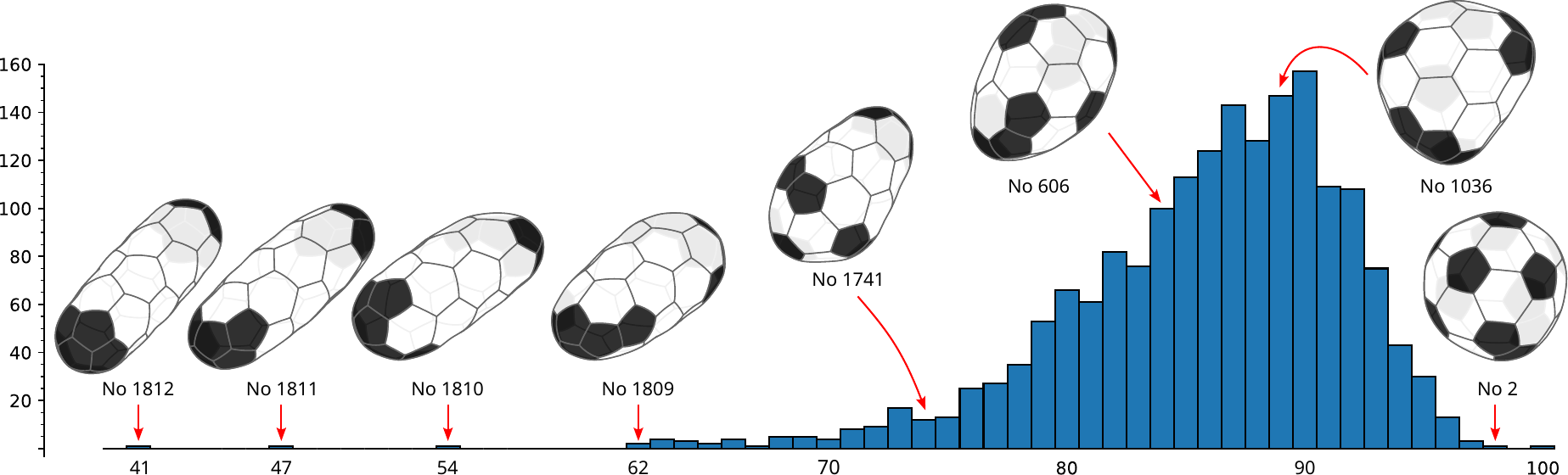}
\caption{
Number of balls depending on the roundness.
No.~2 is the second roundest, No.~606 is the median and No.~1036 the closest to the average roundness.
The distribution is not Gaussian (but why should it be?).
}
\label{fig:roundness}
\end{figure}

One could ask whether the volume estimation is robust.
For example, how does it depend on the number of iterations of the Catmull–Clark algorithm?
We computed the correlation between the volumes: it is larger than $0.99997$ between iterations 1 and 2, and larger than $0.9999994$ between iterations 2 and 3.
We also checked that the roundness of every ball changes by at most $2$ between iterations 1 and 2 ($82\%$ are unchanged), and by at most $1$ between iterations 2 and 3 ($94\%$ are unchanged).
The numbering is more sensitive, however, because many balls have very close volume or roundness.


\begin{figure}[hbtp]
\includegraphics[width=\textwidth]{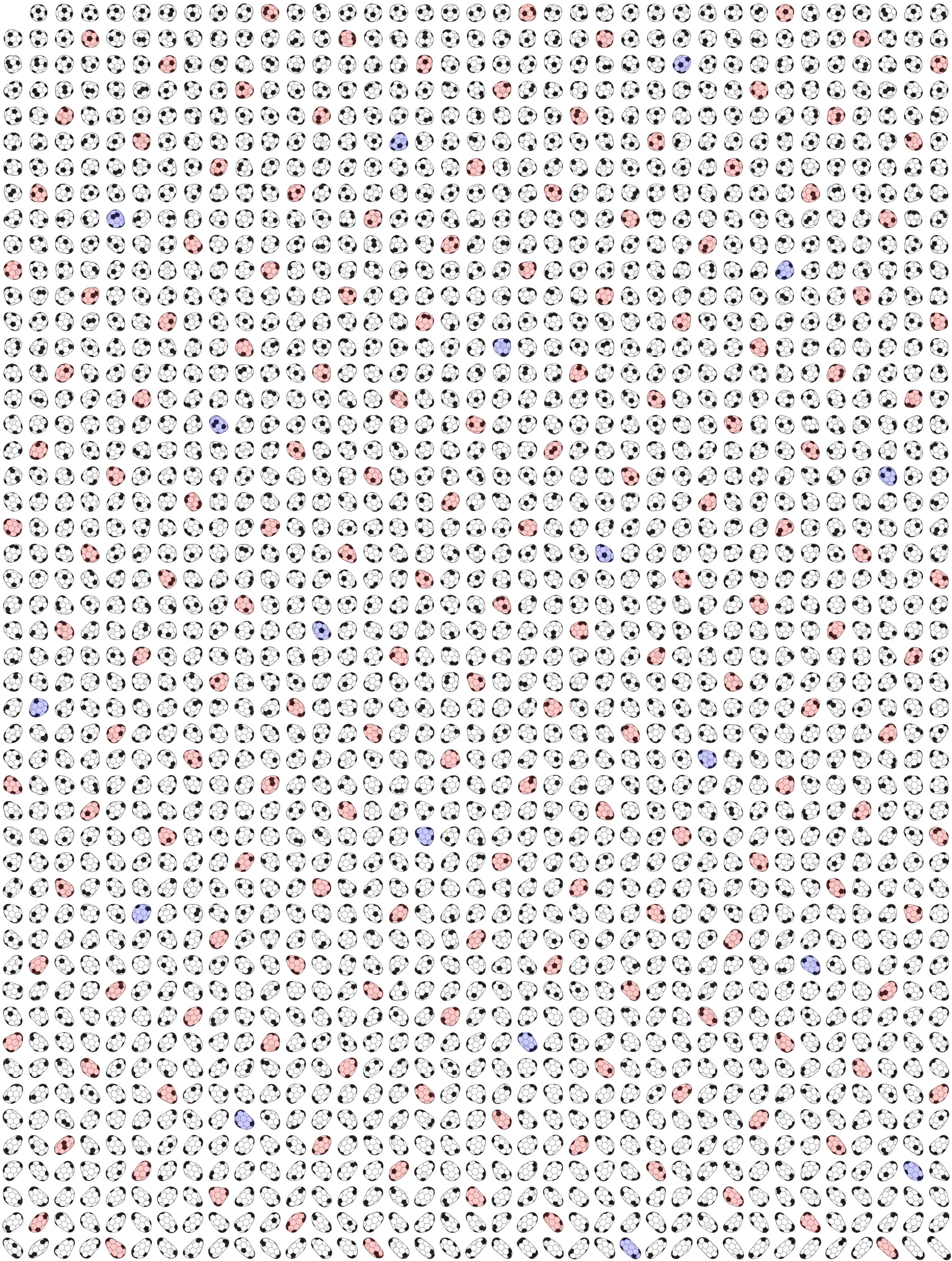}
\caption{
Balls No.~1 through 1812 (in reading order).
Multiples of 100 (resp. 10) are in blue (resp. red).
Can you spot the balls depicted in Fig.~\ref{fig:ballons_cirm}?
}
\label{fig:toutes}
\end{figure}

\section{Crafting}
\label{sec:crafting}
Now that we have models of all these 1812 footballs, let us turn to practice.
We describe the crafting process, the first realization (46 balls at the time of writing), as well as the strategy we envisage to bring this project to completion (that is, 1812 balls).

\subsection{Panels}

The panels are made of cowhide leather, as footballs were until the early 1970s.
Since then (Telstar Durlast, 1974), polyurethane has been preferred to leather, because leather absorbs water and does not allow playing in wet weather.
This criterion is obviously less important for our project than the inimitable look and feel of genuine leather.
Although vintage leather footballs are now often made with warm-colored leather (brown, beige), we kept the black and white of the original Telstar as an explicit reference to this iconic ball.
The leather is between 1.2 mm and 1.4 mm thick.
As in a real football, it is reinforced by a {\em lamination}: several layers of technical cotton fabric are glued on its back with latex.
This limits the deformation of the leather and gives the ball a standard weight (about 410 g — which we achieved with two layers of $245\textrm{ g/m}^2$ — the weight of the glue, difficult to estimate, should not be neglected).

\begin{figure}[hbt]
\centering
\includegraphics[width=\textwidth]{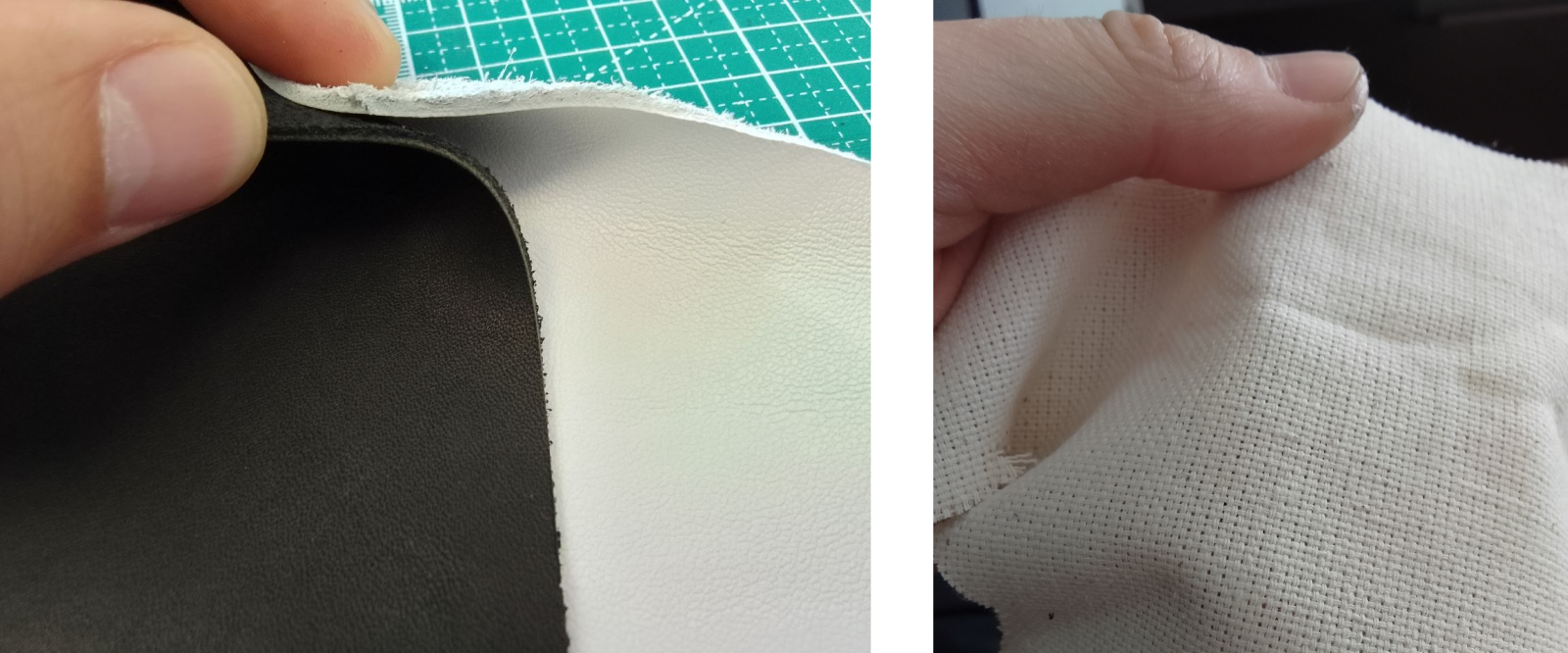}
\caption{
Black and white full-grain leather, cotton reinforcement fabric.
}
\label{fig:lamination}
\end{figure}

The panels are then cut with clicker dies (purchased from a professional in Punjab, 2 kg of steel each) using a 10-ton press (Fig.~\ref{fig:cut}).
Under the press, the upper part of the clicker die lowers and cuts the leather.
At the same time, needles that come out through holes in the movable lower part of the clicker die pierce 9 fine holes along each side for the stitching thread.
The pentagons and hexagons have a side length of 48 mm.
Since the stitched part ends up on the inside of the ball, the visible panels of the finished ball are about 45 mm, which yields the standard circumference of 70 cm for the ball.

\begin{figure}[hbt]
\centering
\includegraphics[width=\textwidth]{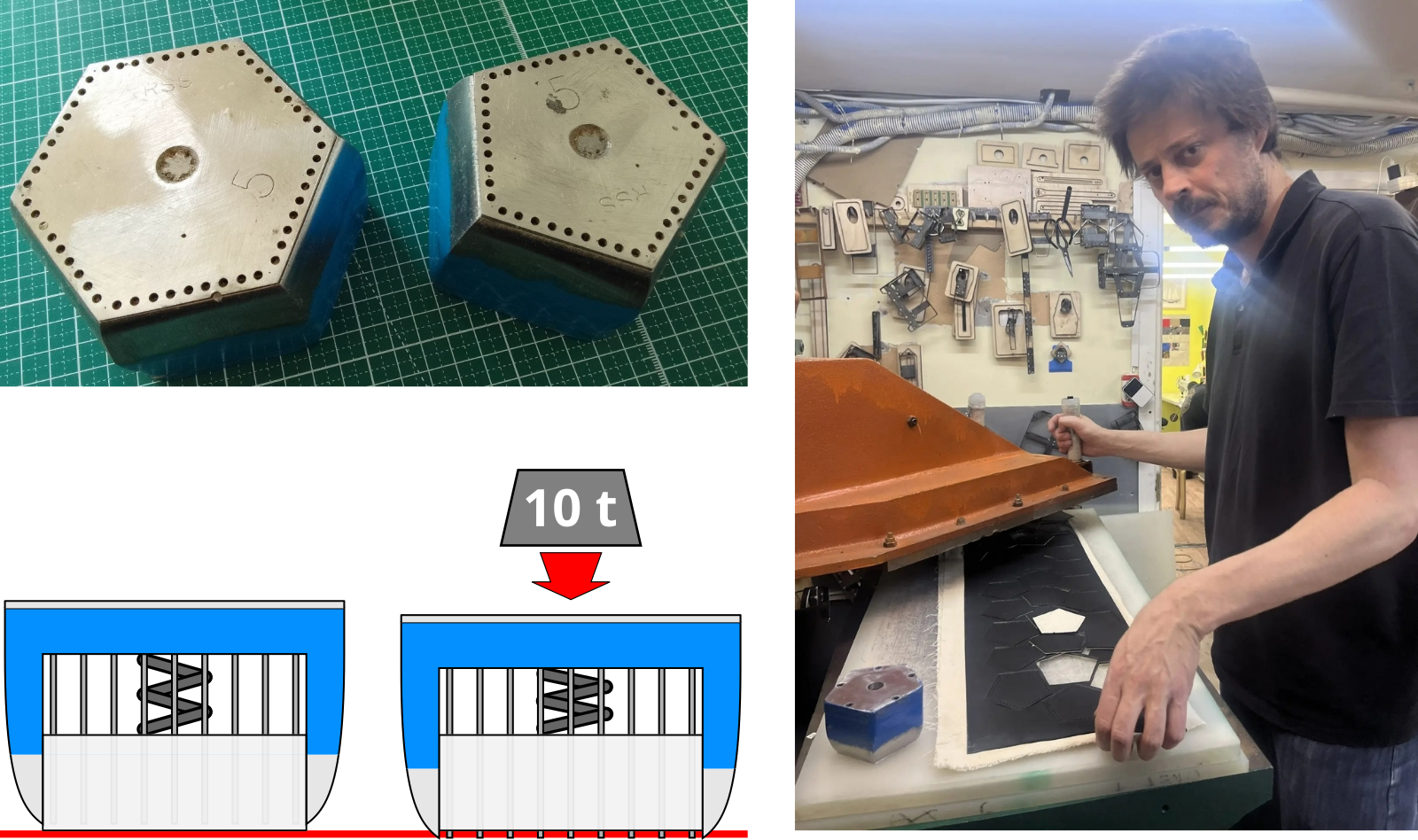}
\caption{
Clicker dies (upside down), scheme of a clicker die, 10-ton press.
}
\label{fig:cut}
\end{figure}

Finally, a hot press and movable brass type allow the panels to be numbered by stamping (Fig.~\ref{fig:stamp}).
The number is stamped on the panel which has the valve, taking care to align the number along the direction in which the leather stretches the least, so as to limit deformation once the ball is inflated.

\begin{figure}[hbt]
\centering
\includegraphics[width=\textwidth]{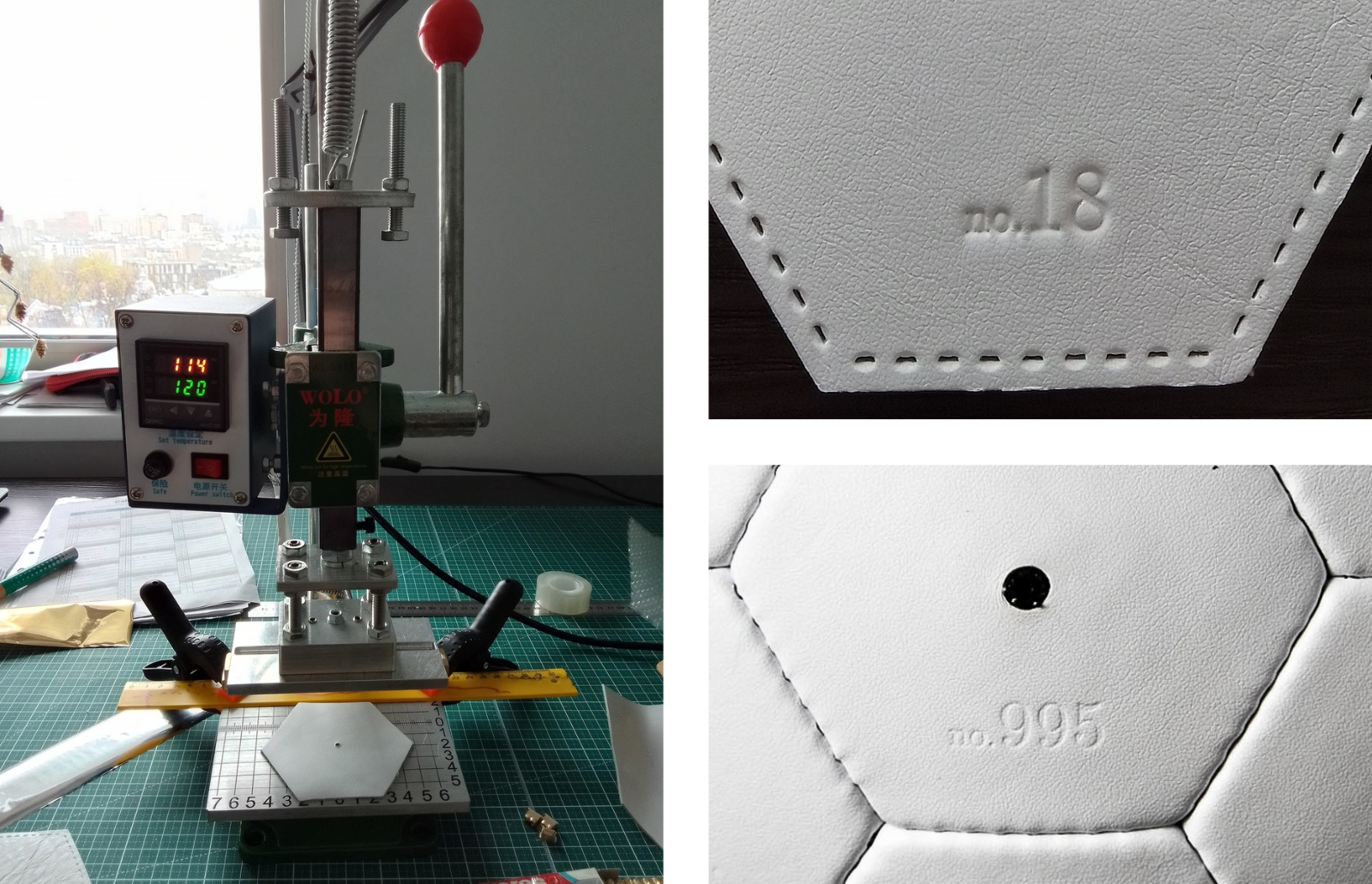}
\caption{
Hot press, stamped panels before and after stitching.
}
\label{fig:stamp}
\end{figure}

\subsection{Stitching}

Once the panels are made, they must be assembled.
We use the saddle stitch, a simple and classic leatherworking stitch known for its strength: two needles at the ends of a single thread pass through the holes in a zigzag, one opposite the other (both needles are inserted face-to-face into each hole at the same time so that one needle does not pass through the thread attached to the other).
After each pass through a hole, the thread is pulled very firmly from both sides so as to tighten the two panels as much as possible against each other, so that the thread does not show when the ball will be turned inside out and inflated.
Figure~\ref{fig:stitchA} illustrates the first stitch and Figure~\ref{fig:stitchB} shows how to ``turn'' to continue the next stitch and how to end it.

\begin{figure}[hbt]
\centering
\includegraphics[width=\textwidth]{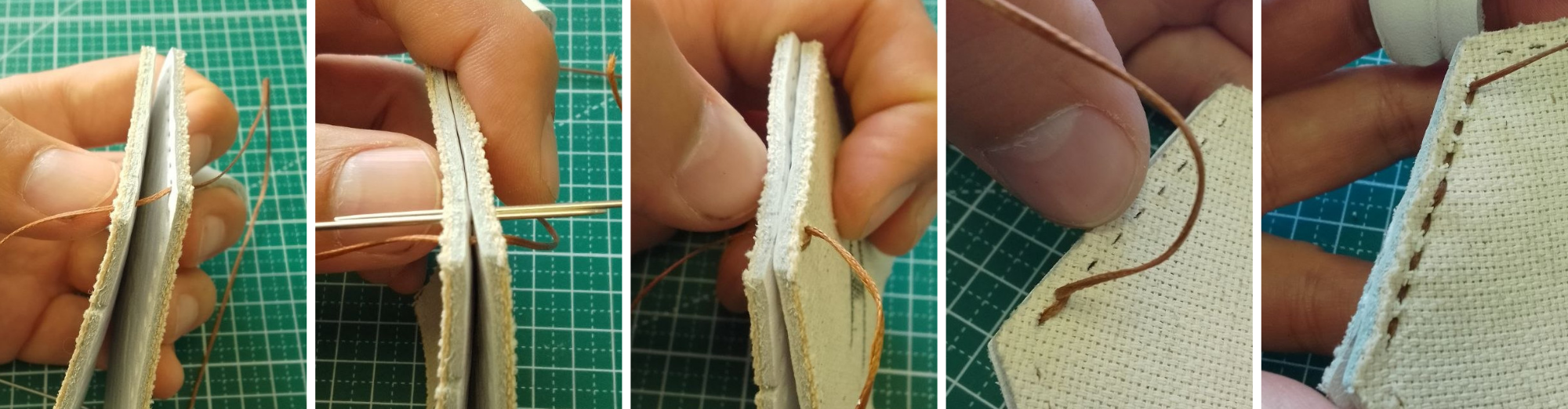}
\caption{
Stitching the first edge.
}
\label{fig:stitchA}
\end{figure}

\begin{figure}[hbt]
\centering
\includegraphics[width=\textwidth]{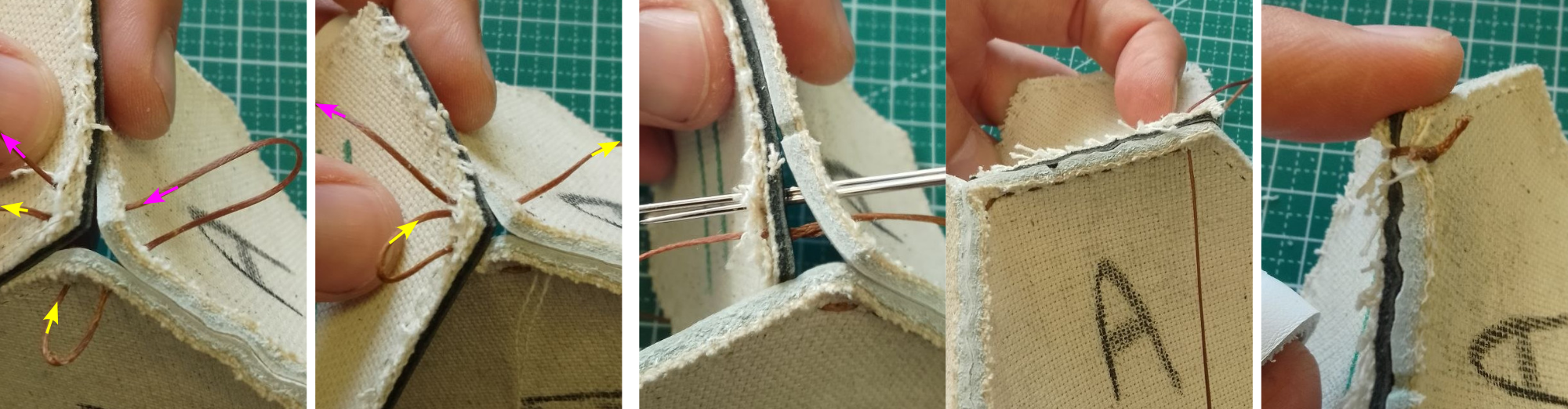}
\caption{
Stitching the next edge, ended with a double half-hitch knot.
}
\label{fig:stitchB}
\end{figure}

The ball is first assembled inside-out, because this allows the needles to go straight through the matching holes of two panels pressed tightly together with their outer faces against each other.
When only a few edges remain to be sewn (typically 4 or 5), the ball is turned right-side out (this requires some effort) — not forgetting to have previously glued the inner tube to the inner face of the panel that has the valve.

The last stitches take longer.
Indeed, you cannot enter the ball to finish it from inside.
The needles must be passed through the matching holes of the two panels independently, and the thread is not really pulled so as to let space between the panels for passing the needles.
Figure~\ref{fig:stitchC} illustrates that.

\begin{figure}[hbt]
\centering
\includegraphics[width=\textwidth]{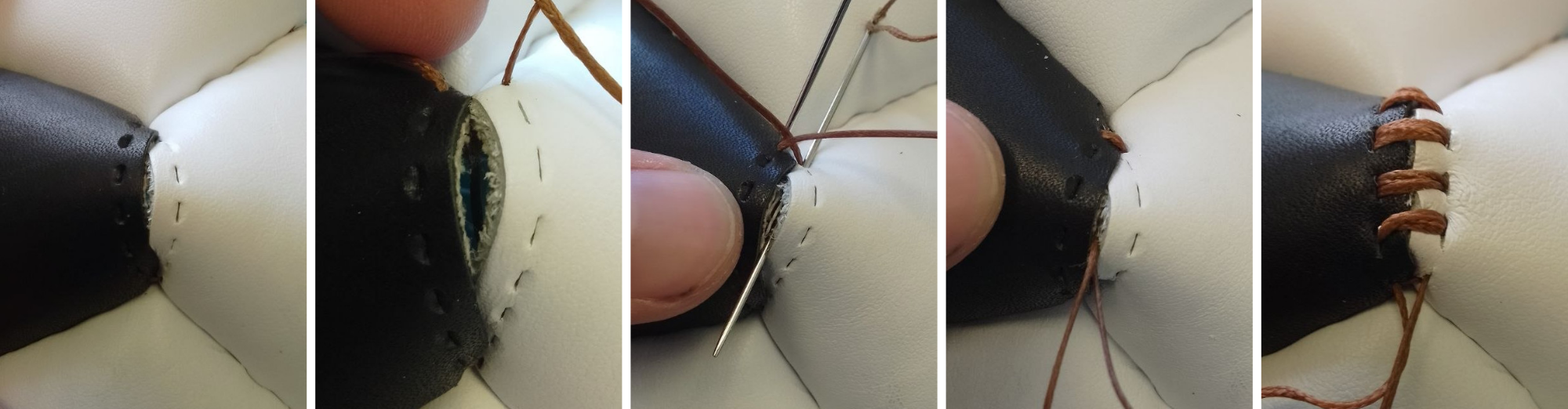}
\caption{
Stitching the last edge with a loose thread.
}
\label{fig:stitchC}
\end{figure}

Now the magic will happen.
The thread is pulled very firmly so as to tighten all the stitches along the edge.
This is difficult, one can use an awl to tighten stitch by stitch.
The panels come into contact and the seam disappears inside the ball!
It only remains to end the stitch and make something with the threads.
An elegant method is to bring the last vertex to a vertex on the opposite side of the ball, pass the needles through so that they come out at this opposite vertex (do not pierce the bladder!), pull very firmly, and tie a double half-hitch knot which, when the ball is inflated, will slip inside the ball, just as if it had been tied by a little elf inside.
Figure~\ref{fig:stitchD} illustrates these last steps.

\begin{figure}[hbt]
\centering
\includegraphics[width=\textwidth]{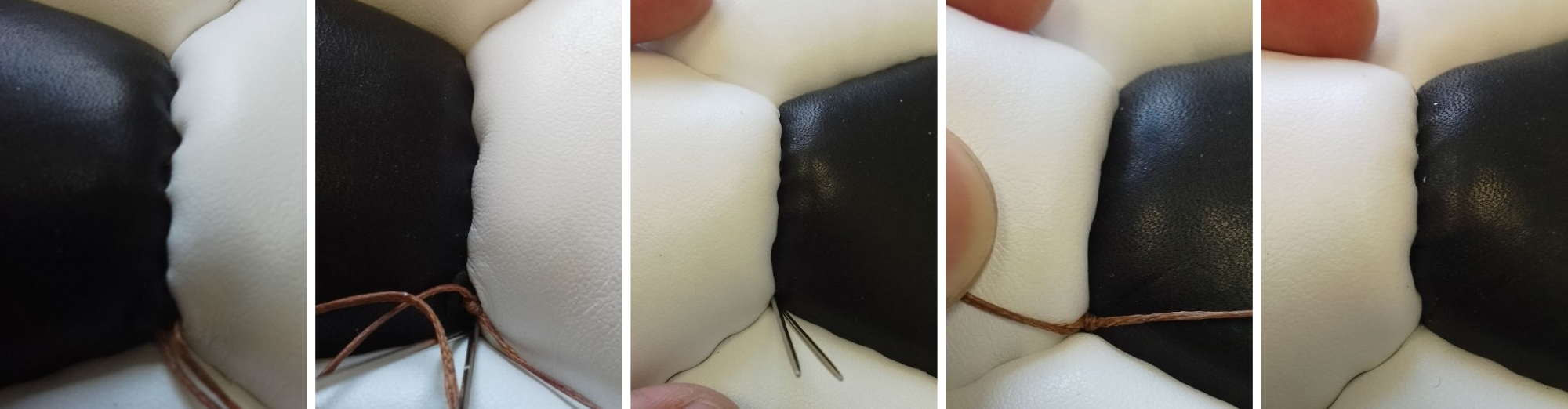}
\caption{
Bringing the last panels into contact and tying the final knot.
}
\label{fig:stitchD}
\end{figure}

All in all, this takes about 2 hours for a professional to stitch a ball (my personal record is 10 hours -- in addition to dexterity, it requires strong fingers to insert and especially to pull the needles).

\subsection{Nets}

We have so far omitted the question -- absolutely crucial in our project -- of the order in which the panels are sewn.
The problem is to describe each ball in a way that makes it as easy as possible to assemble.
A natural idea is to use a {\em net}, that is, a flat, two-dimensional pattern of connected polygons that can be folded to form the wanted ball.
A historical example for the truncated icosahedron (the classic football) is the net given in the year 1525 by Albrecht Dürer in \cite{Due25}, Fig.~\ref{fig:durer} (amusingly, Dürer counted 2 extra vertices).

\begin{figure}[hbtp]
\centering
\includegraphics[width=0.7\textwidth]{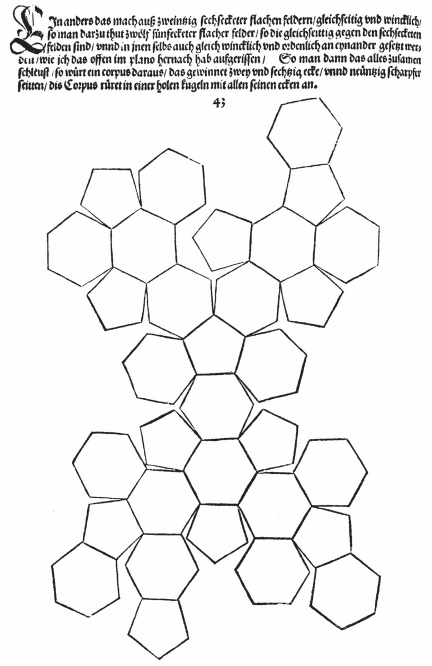}
\caption{
{\small
{\em
For the other [solid], make it from twenty planar hexagonal faces, equiangular and inclined, then add twelve planar pentagonal faces, which are equal with respect to the hexagonal faces, and join them together in the same inclined and orderly manner, as I have drawn it openly afterward in the plan.
When all this is closed together, a body results that has sixty-two vertices and ninety sharp edges.
This body touches the inside of a hollow sphere with all its vertices}
(caption translated from Early New High German by DeepSeek).} 
}
\label{fig:durer}
\end{figure}

Such a net where the pattern is folded along polygon edges is called an edge-unfolding.
Surprisingly, the question of whether a polyhedron admits an edge-unfolding, often referred to as the Dürer problem, is still open \cite{ORO11,DO07}.
Of course, one could always proceed as follows.
First, consider the graph whose vertices are the face centers and whose edges connect vertices that correspond to adjacent faces, and remove as many edges as possible so that the graph remains connected: this yields a so-called {\em spanning tree} of the {\em dual graph} of the polyhedron.
Then, edge-unfold along this spanning tree, that is, cut the polyhedron edges between faces corresponding to vertices not connected by an edge of the spanning tree.
However, different branches of the spanning tree may overlap after unfolding and a suitable spanning might not exist.
This question is open even if the polyhedron is assumed to be convex or if it is a polyomino, that is, unit cubes glued along entire faces.
To the best of our knowledge, the case of a fullerene is also still open:

\begin{question}
Does every fullerene admit an edge-unfolding?
\end{question}

We checked exhaustively by computer that this holds for fullerenes with 20 hexagons, that is, for our 1812 footballs.
However, the obtained nets are not always ideal, because two panels that have to be sewn together can be represented in the net by faces which are very far apart.
For example, can you spot all the five hexagons that have to be sewn to the bottommost pentagon in Fig.~\ref{fig:durer}?

Probably for that reason, ball factories generally use the two-part net depicted in Fig.~\ref{fig:classic_net}, at least when the panels are decorated so as to form a nice pattern over the whole ball (the case of black and white panels is indeed much easier: one checks that it suffices to never sew together two pentagons or three hexagons).
Each part of the net describes one hemisphere, with a special mark to indicate where the hemispheres have to be finally sewn together.
On each part, it is quite easy to guess along which edges any two faces must be sewn.

\begin{figure}[hbt]
\centering
\includegraphics[width=0.9\textwidth]{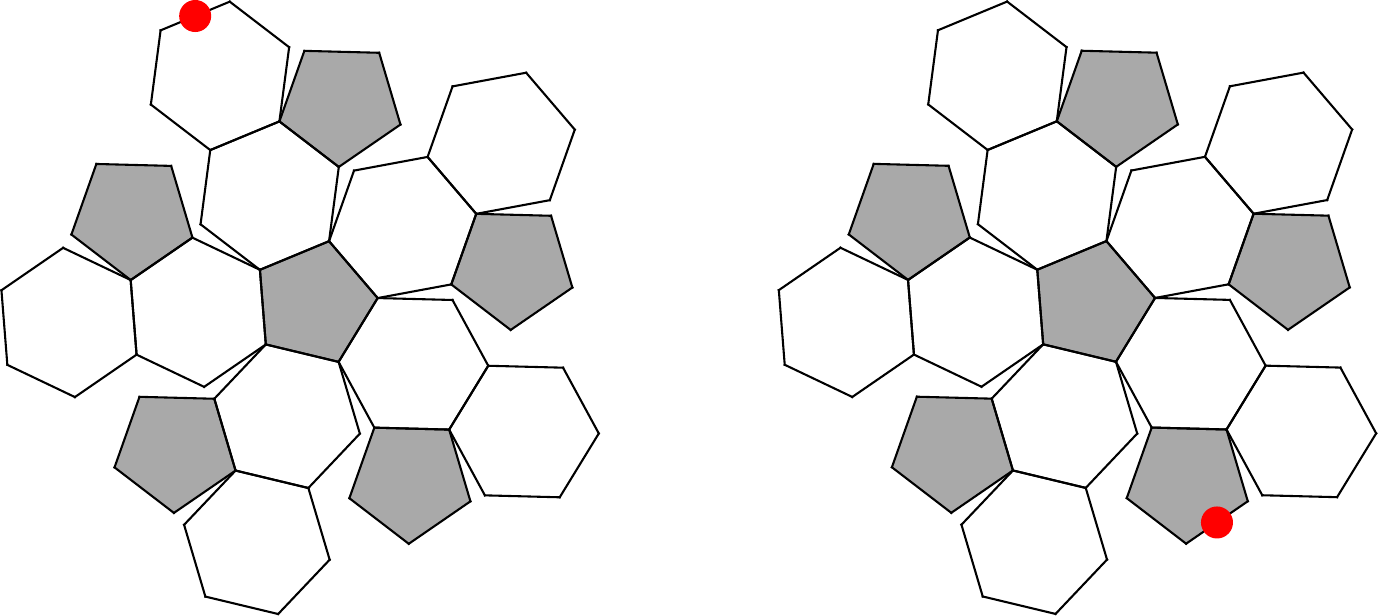}
\caption{
A two-part net for the classic football.
}
\label{fig:classic_net}
\end{figure}

We mimic this two-part net approach for our 1812 balls.
Namely:
\begin{enumerate}
\item the model obtained in Section~\ref{sec:modeling} is rotated so that the X axis is along the greatest dimension of the ball and Y along the second one;
\item the 16 panels with the largest Z form the northern hemisphere, the other 16 ones the southern hemisphere;
\item in each hemisphere, the panels are ordered so that every panel (except the first two ones) is adjacent to 2 or 3 previous panels (we always succeed in finding such an ordering -- coincidence or general fact?);
\item the panels are numbered from 1 to 16, respecting that order and choosing always among the possible panels the closest to the first panel.
\end{enumerate}
The third point aims to ease the work of the stitcher, who usually stitches two or three edges before knotting the thread.
The other points seek to yield a ``compact'' net in order to visualize easily along which edges any two faces must be sewn.
The red circle around a number indicates the placement for the valve: it has been (automatically) chosen in the flattest zone of the ball in order to allow the bladder to expand more freely.
Figure~\ref{fig:nets} shows examples of nets obtained in that way.
In practice, numbers are handwritten on the back of each panel to ease the stitching.

\begin{figure}[hbt]
\centering
\includegraphics[width=\textwidth]{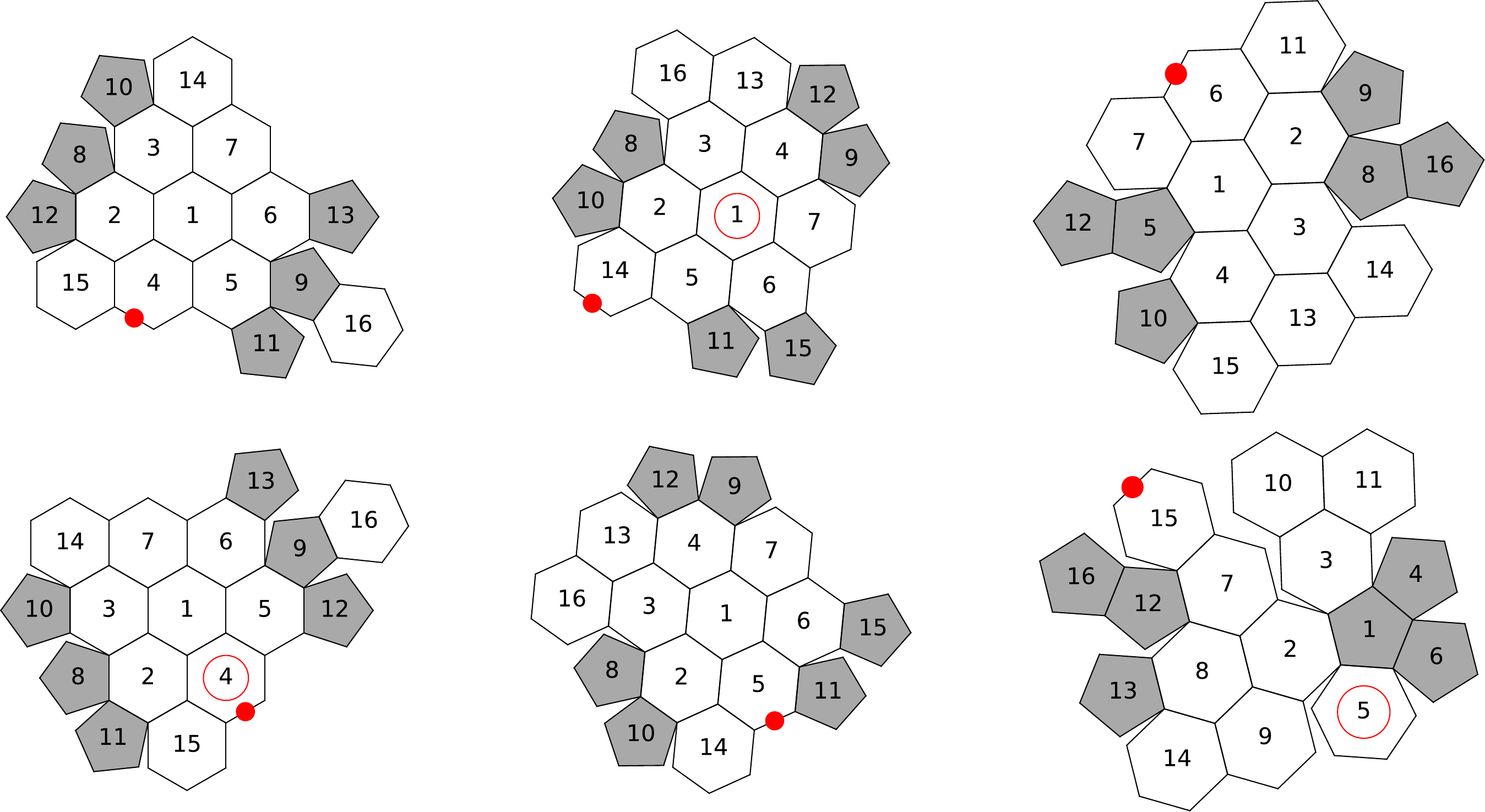}
\caption{
Two-part nets for the balls depicted in Fig.~\ref{fig:ballons_cirm} (from left to right, with the north hemisphere on top of the south one).
}
\label{fig:nets}
\end{figure}

\section{Genesis and perspectives}

This project has taken me far beyond the academic world of mathematics that is my own, and has led me to meet many different people who have helped me to varying degrees.

The project was born in October 2023, when I came across — I no longer remember how — the website of Vista Ballon, a Marseille-based company that offered hand-sewn balls with custom designs.
I submitted them my idea of sewing the pieces differently, without really believing in it, but I quickly received a positive reply from Jean-Baptiste de Tourris and Frédéric Caens. 
I thus turned to the generation of fullerenes, spotted three that I liked, and hand-drew three nets which I first tested (Fig.~\ref{fig:paperfold}) before sending them to Vista.
This materialized in March 2024 as the three balls shown in Fig.~\ref{fig:ballons_cirm}.

I showed these balls at the Centre International de Ren\-con\-tres Ma\-thé\-ma\-ti\-ques (CIRM) in Marseille, thanks to its communications officer, Stéphanie Vareille, which caught the attention of my first client: the Palais de la Découverte in Paris, in the persons of Noëlle Krajcman, Guillaume Reuiller and Robin Jamet.

There followed a period of doubt: should I continue beyond these three copies? Make all 1812? With what funding? Have them made: where? Make them myself: how? The encouragement (and orders) of Sre\v{c}ko Brlek, then Nicolas Bédaride and Damien Jamet, were precious at that time, as was the abundant advice of Vincent Delieutraz (which I more or less followed).
First, with the help of Gyunay Akhmedova, I got in touch with the sewing atelier of Aysel Özcan in Burdur, Türkiye. 
They made me 8 more test balls in early 2025.

However, I gradually became convinced that it would be more convenient to achieve autonomy if I were to continue this project, which seemed likely to be a long-term endeavor.
With the help of Olga Sizova and Nadezhda Kalinina (Mjöllborg, Mytishchi) and the advice of Dmitri Arkhipov (Carball Smolensk) and Igor Timoshchuk (HSE, Moscow), I eventually worked out how to manufacture my own panels.
My attempt to hand-stitch a ball -- about 12 hours for a very average result -- made me understand that I would need, at least initially, the help of professional stitchers.
Where to find them?

A quick search taught me that 70\% of hand-stitched balls are made in the city of Sialkot, Pakistan.
I therefore looked for a contact there, and found one in the persons of Farhan \& Suleiman Khawaja (BJ Sports), who were very responsive and enthusiastic (even if my strange and modest demand represented but a drop in the ocean compared to the roughly 3000 balls their fabric produces daily).
I produced a first set of 1472 panels, with the help of Evgeniy Poloskin, packed them in my bag and set off for a week in Pakistan, in May 2026.
There, I was able to observe the hand-stitching of the 46 balls made possible by these 1472 panels and learn the techniques, before returning with a big cardboard box of (deflated) balls in my arms (Fig.~\ref{fig:46balls}).

\begin{figure}[hbt]
\centering
\includegraphics[width=\textwidth]{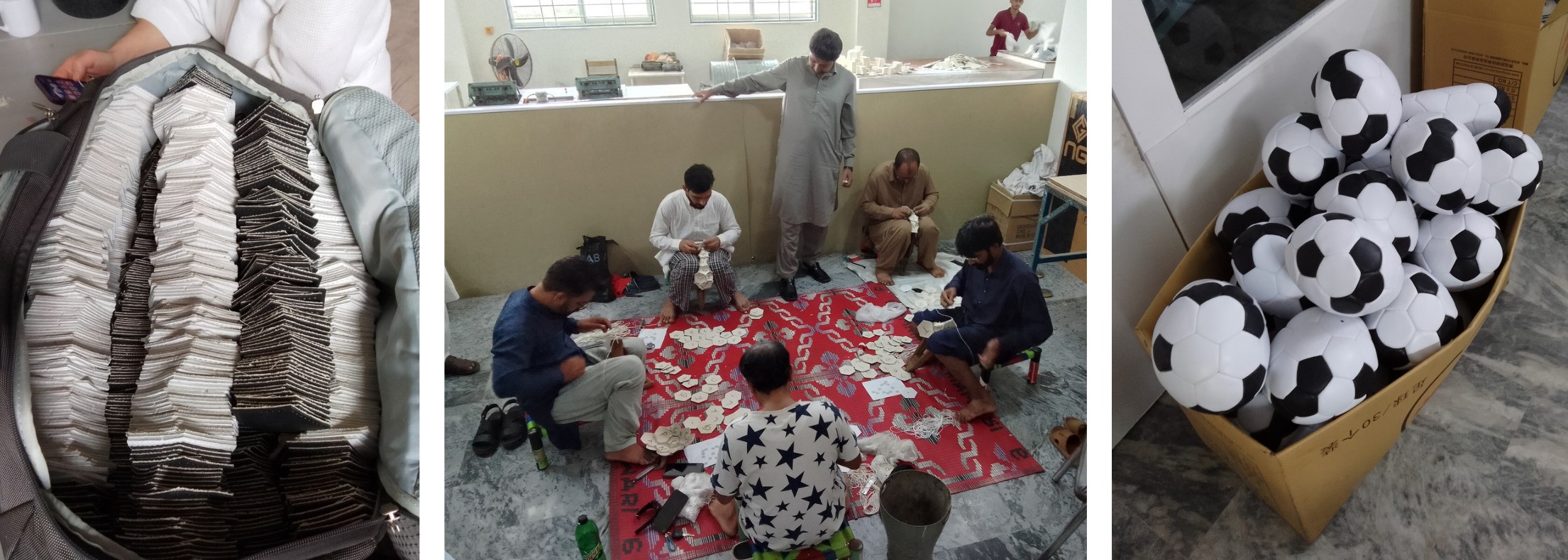}
\caption{
Transformation of 1472 panels into 46 different balls.
}
\label{fig:46balls}
\end{figure}

I then sell these balls here and there, while keeping some aside as demonstration copies.
I would like to thank here all my first buyers, who almost covered my costs and thus allowed me to think about what would come next.
In chronological order: Noëlle Krajcman, Guillaume Reuiller, Robin Jamet, Damien Jamet, Sre\v{c}ko Brlek, Nicolas Bédaride, Nazim Fates, Romain Couë\-til, Doria Fouchel, Andres Navas, Oleg Rodionov, Delphine Bournay \& Thierry Coulbois, Idrissa Kaboré, Quentin Landrivon, Raoul Santachiara, Matthew Parker, Faye Goldman and François Delannoy.

So, what next?
The diffusion strategy is still unclear to me, torn between the worlds of mathematics, football and art.
The production strategy as well: produce in small batches over a long time or make everything at once?
At least the manufacturing process is under control: based on the 46 balls produced, I estimate that it takes on average about 10 minutes to manufacture the panels needed to sew one ball, to which must be added a good 2 hours of stitching (an operation that can, however, be massively parallelized in Pakistan).
For now, I am considering a hybrid strategy: on the one hand, to continue producing at a slow pace (about a hundred per year) the series of numbered leather balls; on the other hand, to look for a sponsor for a ``pool of ball''.
The idea is to have an entire second series made of copies in synthetic leather, without a stamped number, for a traveling installation where these balls would be placed in a garden pool that the public could enter.
Indeed, based on the average density of a random packing of spheres -- $64\%$ according to the experiments carried out in \cite{BJ60} -- the volume occupied by 1812 standard balls is less than $16\textrm{m}^3$.
Our 1812 balls are not round, but their volume is smaller than a standard ball, and the sphere is moreover conjectured to be the convex solid that fills space the least efficiently (Ulam's conjecture \cite{Kal14}).
So they should quite easily fit in $16\textrm{m}^3$.
This is about the volume of a common round tubular garden pool 4.50 m in diameter and one meter deep.
And also the volume of an ordinary delivery van.
With a total weight of about 750 kg for the 1812 balls and a demountable pool, a transportable installation is therefore conceivable.
One could consider starting with balls 3D-printed in PLA at a 1:5 scale, in order to have a model to show.
Rough tests suggest that about 15 balls could be printed at once in about one night, and therefore the 1812 balls in about 4 months.
All this for a volume of about 130 L -- an inflatable garden paddling pool -- and an estimated weight between 25 and 30 kg -- relatively easy to transport.


\appendix
\section{Characterization}
\label{sec:characterization}

Since these 1812 footballs are all unique, one could theoretically describe each by a property of its own.
As we shall see, this turns out to be more or less easy depending on the ball: some are more unique than others!
In this appendix, we examine a few remarkable properties.

\subsection{Roundness}

Of course, the first property that comes to mind is roundness, since we numbered the balls according to this criterion.
But how to measure roundness only by visual inspection from a ball in hand?
Let us give a combinatorial approach that uses the notion of curvature.
The {\em curvature} at a point of a surface is defined as the sum of the angles of an infinitesimal triangle around this point minus $180^\circ$.
This definition is due to Alexandrov and generalizes the Gaussian curvature \cite{Ale05}.
Here, every point on a panel has zero curvature everywhere because the sum of the angles of a triangle drawn on it is equal to $180^\circ$.
Bending the panel does not change anything: it must be stretched in order to change its {\em intrinsic} metric.
It also holds on the edge between two panels because a crease is only a particular kind of bending.
But the Gauss–Bonnet theorem states that every topological sphere -- our footballs are such objects -- always has a total curvature of $720^\circ$.
This leaves no choice: all the curvature must be concentrated at the vertices.
When three hexagons meet at a vertex, the sum of the angles of any triangle around this vertex is $180^\circ$ and the curvature is thus still zero at such a vertex.
But when a pentagon joins, since its angles are only $108^\circ$ (instead of $120^\circ$ for the hexagon), it yields an {\em angular deficit} of $12^\circ$, which is reflected in the sum of the angles of any triangle around the vertex, hence in the curvature.
This yields four different curvatures depending on the number of pentagons that meet at this vertex (Fig.~\ref{fig:angle_deficits}).
Note that since every pentagon has 5 vertices, each with an angular defect of $12^\circ$, and since there are always 12 pentagons (Prop.~\ref{prop:12p}), we get a total curvature $12\times 5\times 12^\circ=720^\circ$, as predicted by the Gauss–Bonnet theorem.

\begin{figure}[hbt]
\centering
\includegraphics[width=0.9\textwidth]{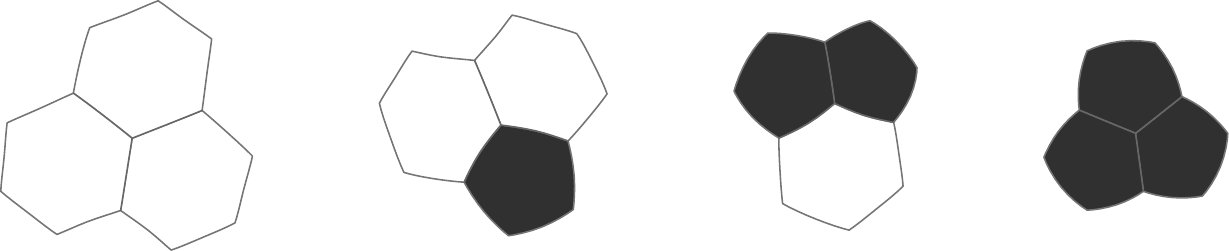}
\caption{
Vertices with curvature (i.e., angle deficit) $0^\circ$, $12^\circ$, $24^\circ$ and $36^\circ$.
}
\label{fig:angle_deficits}
\end{figure}

So, how can curvature be used to approximate the roundness of a ball?
The key idea is that the rounder a ball is, the more uniformly the curvature is distributed over its surface (and conversely).
The roundest ball is the classic one: all 60 vertices have curvature $12^\circ$.
The least round ball (leftmost in Fig.~\ref{fig:roundness}) has 10 vertices of very high curvature ($36^\circ$), 10 vertices of high curvature ($24^\circ$), 10 vertices of ``normal'' curvature ($12^\circ$) and 30 flat vertices.
To measure uniformity, we do not sum the curvatures (the total is always $720^\circ$) but their {\em square}.
The convexity of $x\to x^2$ indeed ensures (Jensen inequality) that the smaller is that sum, the more uniform are the curvatures.
For simplicity, we can divide each curvature by $12$ before taking its square, so that it amounts summing values $0$, $1$, $4$ or $9$, respectively, for the vertices depicted in Fig.~\ref{fig:angle_deficits}.
Actually, since we want the maximum to be reached for the roundest ball, we take the inverse of that sum as our roundness estimator.
A computation shows that the correlation over the 1812 balls between that estimator and the roundness as given by the modeling (Section~\ref{sec:modeling}) reaches $90\%$.
This can even be improved by replacing the curvature of every vertex by an average of the curvatures of the vertices in a neighborhood: a computation shows that, by averaging over vertices within $2$ edges, the previous correlation increases up to $99\%$.
This excellent correlation between values obtained by two completely different approaches (combinatorial {\em vs} modeling) is quite convincing as to the reliability of the roundness estimates used to number the balls.

However, even if it is theoretically possible to compute this value exactly from a ball in hand, it is quite tedious in practice.
Moreover, we get only 18 different values (584 by averaging over vertices within $2$ edges), so it does not help that much to distinguish balls (in \cite{appli} there are 41 different values for the roundness computed from the modeling).
We must thus resign ourselves to abandoning roundness as a simple discriminating property.

\subsection{Spiral}

Another way to describe a ball is the {\em spiral encoding} defined in \cite{MMD91} (see also \cite{FM95}) for fullerenes.
The idea is simple: choose an arbitrary face, and spiral around it as if peeling an orange, as in Fig.~\ref{fig:spiral}.
This yields a sequence of pentagons and hexagons that completely describe the fullerene.

\begin{figure}[hbtp]
\centering
\includegraphics[width=\textwidth]{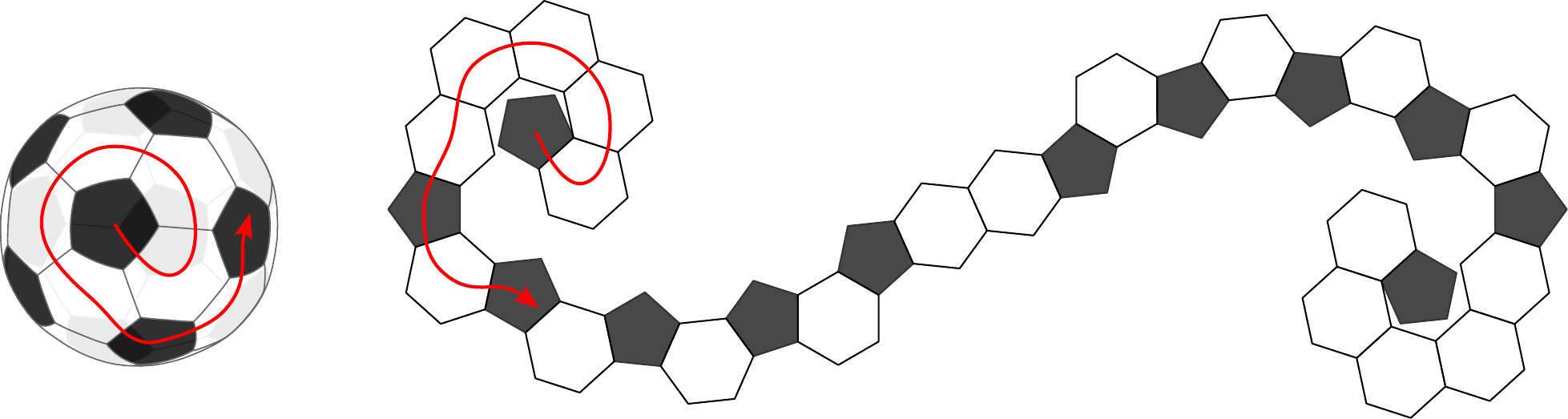}
\caption{
Peeling a ball (here the classic one) as an orange.
}
\label{fig:spiral}
\end{figure}

Unfortunately, it turns out in practice to be difficult to follow a spiral correctly on a ball held in hand (unless you draw the path on the ball with a marker).
Moreover, depending on the starting face and direction of rotation, different codings of the same ball can be obtained.
The solution proposed in \cite{MMD91} is to compute {\em all} codings and then keep the lexicographically minimal one -- a method that seems rather impractical on a ball held in hand.
Also note that this spiral coding does not always exist for large fullerenes -- although this does not affect our 32-faces footballs because the smallest counterexample has 192 faces \cite{FM93}.
Again, we must resign ourselves to abandoning spiral encoding as a simple discriminating property.

\subsection{Signature}

What immediately catches the eye when holding a ball is the dark patterns formed by the clusters of pentagons.
A classic ball, for example, has all its pentagons isolated.
The balls in Fig.~\ref{fig:ballons_cirm} have, respectively, two rings, three identical blocks, and four identical blocks.
There are 33 different patterns (up to isometry) that can be spotted on the 1812 balls, depicted in Figures~\ref{fig:paths} and \ref{fig:blocks}.
Among them, 16 are chiral, that is, different from their mirror image -- both may appear in the same ball (e.g., No.~1694).

\begin{figure}[hbtp]
\centering
\includegraphics[width=\textwidth]{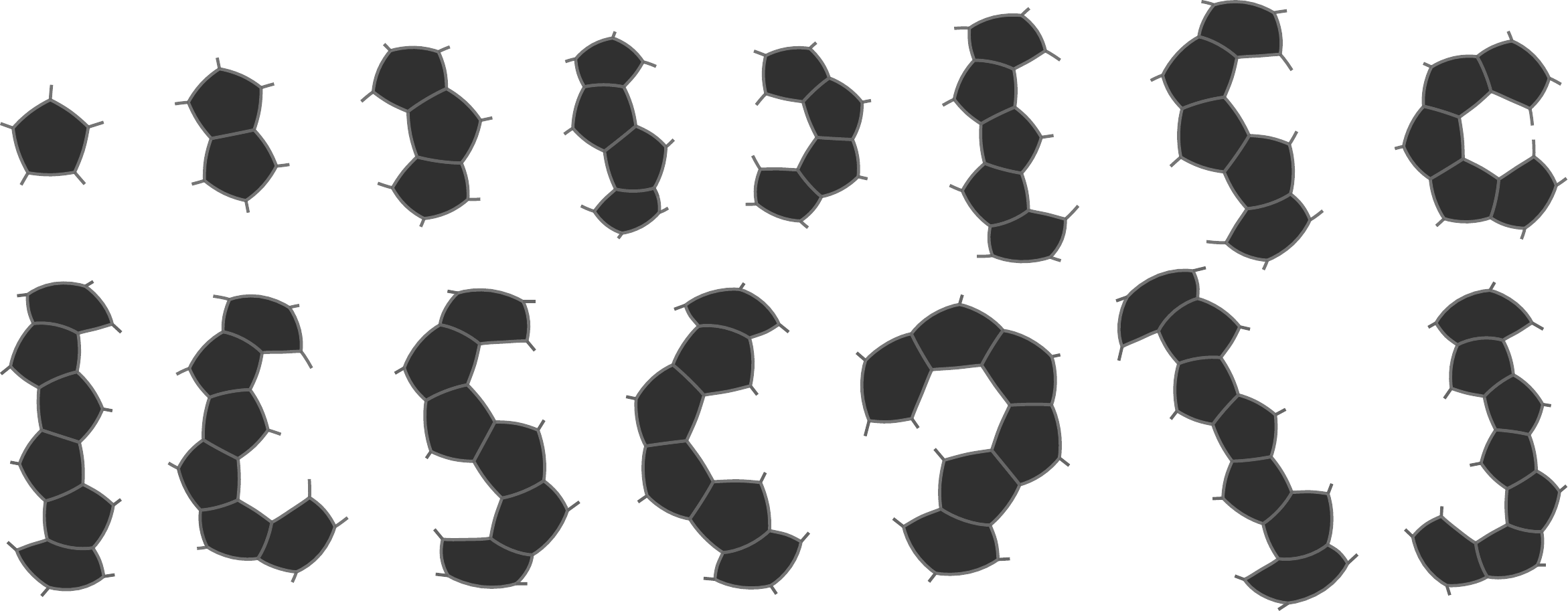}
\caption{
All the different paths that can appear on our footballs.
}
\label{fig:paths}
\end{figure}

\begin{figure}[hbtp]
\centering
\includegraphics[width=\textwidth]{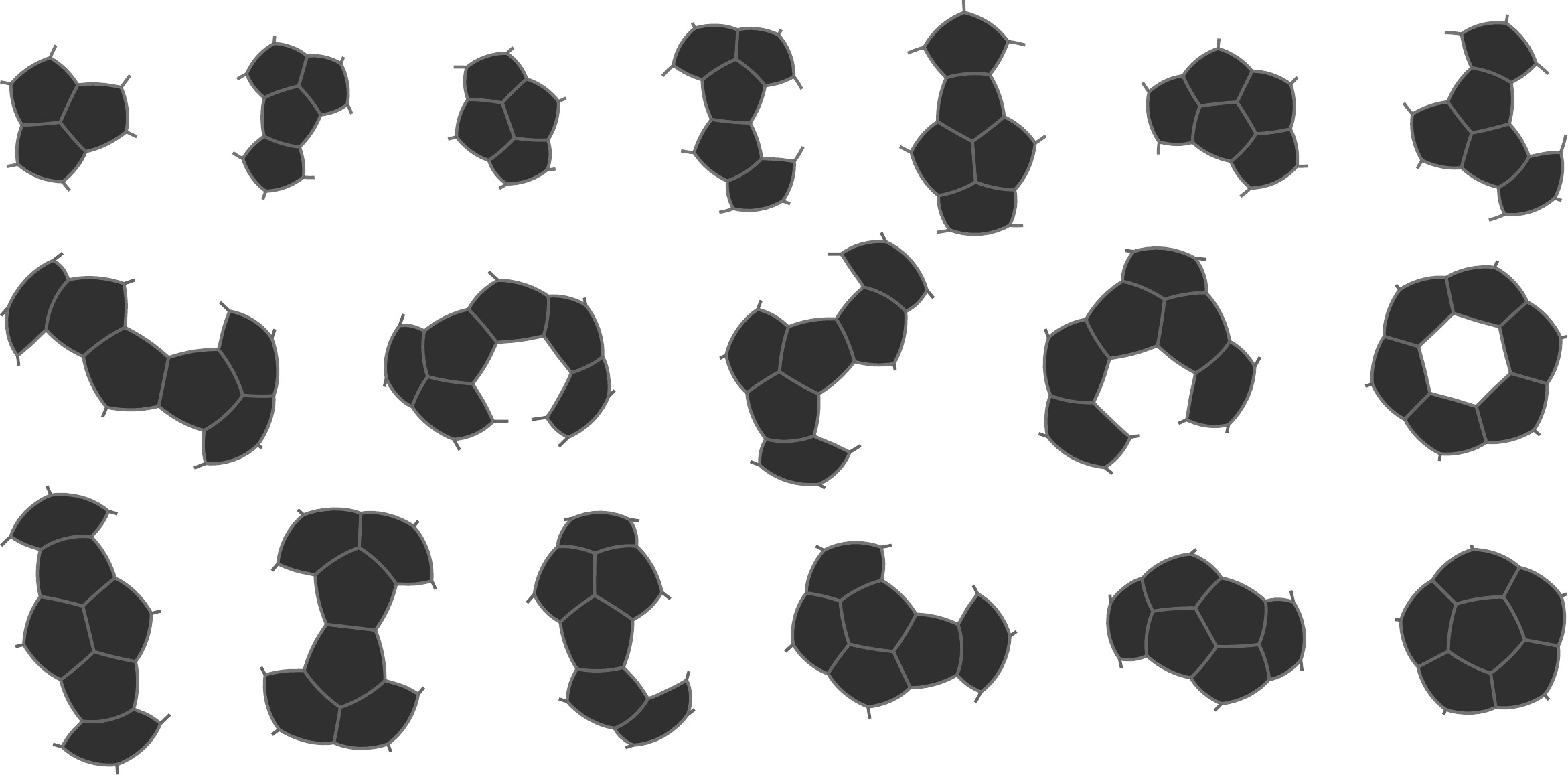}
\caption{
All the different blocks that can appear on our footballs.
}
\label{fig:blocks}
\end{figure}

Some of these patterns are difficult to distinguish, and naming them all is rather cumbersome.
For this reason, we simplified: the patterns in Fig.~\ref{fig:paths} are called {\em paths} and denoted $P_k$, those in Fig.~\ref{fig:blocks} are called {\em blocks} and denoted $B_k$ -- in both cases $k$ denotes the number of pentagons in the pattern.
We thus distinguish only 11 patterns: the paths $P_1,\ldots,P_7$ and the blocks $B_3,\ldots, B_6$.
We then call  {\em signature} of a football the formal sum of its patterns.
For example, the classic ball has signature $12P$, while the three balls in Fig.~\ref{fig:ballons_cirm} have, respectively, signatures $2B_6$, $3B_4$ and $4B_3$; the second roundest one, No.~2, has signature $8P+2P_2$, and the largest cluster of pentagons, $P_7$, appears only in No.~1742 and No.~1780.

Sadly enough, the 1812 balls yield only 133 different signatures, so that a signature usually does not uniquely determine a ball.
There are actually 15 balls characterized by their signature, but as many as 87 balls share the signature $2P+2P_2+2P_3$; Fig.~\ref{fig:signatures} provides a histogram.
Nevertheless, in practice, the signature is quite useful: it allows us to distinguish two random balls in $98.46\%$ of the cases.

Note that keeping track of all the 33 different patterns would not have changed that much: the balls would yield 335 different signatures, 118 of which determine a unique ball, but the above-mentioned 87 balls still share the same signature, and the probability of distinguishing two random balls only slightly improves to $99\%$.

\begin{figure}[hbt]
\centering
\includegraphics[width=\textwidth]{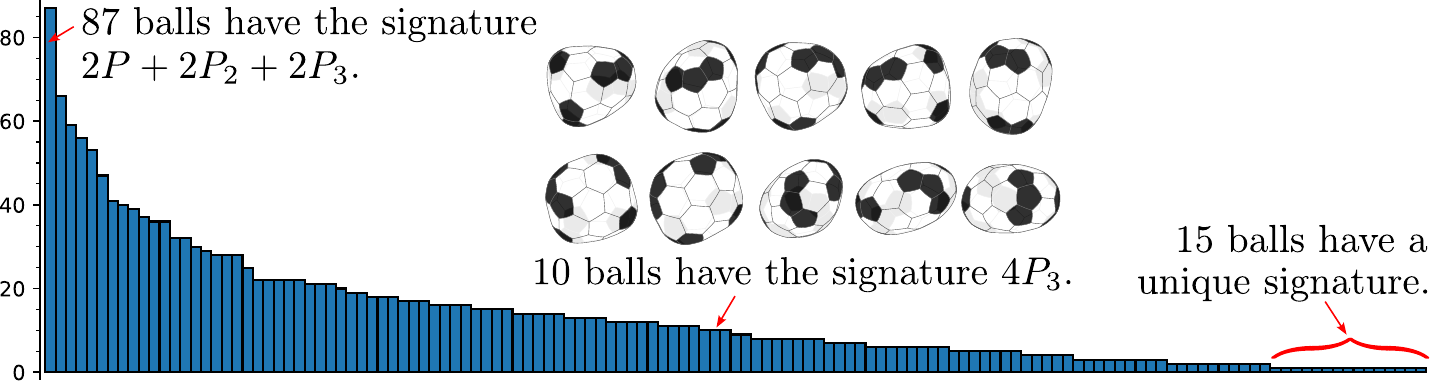}
\caption{
The number of balls corresponding to each of the 133 signatures.
}
\label{fig:signatures}
\end{figure}

\subsection{Symmetries}

Another reason why some balls catch the eye more than others is their symmetries.
A {\em symmetry} of an object is a transformation that leaves it invariant.
In the case of our footballs, these are the isometries: rotations and mirror images.
The set of symmetries of an object is called its {\em symmetry group}.
So, what are the symmetry groups of our 1812 balls?

The problem is that these symmetries are very sensitive to the embedding.
Move a vertex slightly here, stretch an edge a bit there, and your beautiful symmetry disappears.
Not to mention that the vertex coordinates are typically stored as floating-point numbers on the computer. 
One therefore needs some tolerance on the vertex coordinates.
Several programs to do that exist, though their results are not always consistent.
We checked our 1812 embeddings with two programs, PointGroup \cite{pointgroup} and PyMatGen \cite{pymatgen}: they agree on all but 40 footballs (and this is rather robust to modifying the tolerance of the programs).
We counter-checked these 40 cases visually, using the applet \cite{appli}.
Most of the time, one of the programs misses a symmetry; only Ball No.~3 causes trouble to both programs.
Results are given in Table~\ref{tab:sym}, using the Schönflies notation for three-dimensional symmetry groups (the most usual one in chemistry, see e.g. \cite{Vin01} for a detailed account).

\begin{table}
\centering
\setlength{\tabcolsep}{4pt}
\begin{tabular}{|ccccccccccccccc|}
\hline
$C_1$ & $C_2$ & $C_s$ & $D_2$ & $C_{2v}$ & $C_{2h}$ & $S_4$ & $D_3$ & $C_{3v}$ & $D_{2h}$ & $D_{2d}$ & $D_5$ & $D_{5d}$ & $D_{6h}$ & $I_h$\\
\hline
1508 & 189 & 67 & 19 & 9 & 4 & 2 & 3 & 1 & 1 & 4 & 1 & 1 & 2 & 1\\ 
\hline
\end{tabular}
\caption{Number of footballs for each symmetry group.}
\label{tab:sym}
\end{table}

Actually, we got exactly the same statistics as those given in \cite{BD97} for fullerene embeddings, and only 4 differences with those given in \cite{SHSG17}.
This is a bit surprising, because chemists are interested in modeling the structure of a molecule made of atoms and bonds while we are interested in modeling a football made of solid panels of leather.
Are footballs unwittingly doing molecular mechanics?

Let us mention that, whereas the isometries depend on the very embedding, the symmetries of the abstract graph itself are permutations of vertices that are non-ambiguous and can be determined automatically.
They form what is called the {\em automorphism group} of the graph.
However, the correspondence between the automorphism group of the graph and the symmetry group of an embedding of this graph is not straightforward.
A theorem of Mani \cite{Man71} ensures that every fullerene (actually, every planar 3-connected graph) admits an embedding whose symmetry group is isomorphic to its automorphism group.
But this embedding could have nothing to do with ours (in particular, faces can be stretched in Mani's theorem).
In our case, however, our embedding turned out to always suit (with a slight tolerance on symmetry).
This nevertheless raises the question:
\begin{question}
Does every fullerene define an Alexandrov polyhedron whose symmetry group is isomorphic to the fullerene automorphism group?
\end{question}
Another issue is that different symmetry groups can be isomorphic to the same automorphism group.
In our case, both $C_2$ and $C_s$ are isomorphic to the automorphism group $C_2$; both $D_3$ and $C_{3v}$ to $S_3$; and $C_{2v}$, $C_{2h}$ and $D_2$ to $C_2\times C_2$.
We thus have to really look at the embedding to determine which of these isomorphic symmetry groups is the right one.

As for the problem of characterizing each ball, symmetry helps rather little: only 75 balls are characterized by their signature and their symmetry group.
This is not surprising, since there are only 15 symmetry groups and nearly 83\% of the balls have no symmetry at all (group C1).
For example, among the 87 balls with signature $2P+2P_2+2P_3$, 61 have no symmetry.

\subsection{Further}

What other simple properties can one imagine to differentiate these balls?

We attempted to consider {\em H-run}, defined as a chain of ``aligned'' hexagons that one can traverse, leaving each one through the edge opposite to the one through which one entered it.
For example, the classic ball has only H-runs of length 2.
Some H-runs close on themselves into {\em H-loops}: Ball No.~1796 (left in Fig.~\ref{fig:ballons_cirm}) has three H-loops of length 6 (and 12 H-runs of length 3).
Listing all the H-runs of a ball held in hand is difficult, but finding the longest one is usually fairly simple.
The longest H-run has length 21 (Ball No.~1405); it self-intersects since there are only 20 hexagons in the ball.
The longest H-loop has length 14 (Ball No.~1579).
Combining the lengths of the longest H-run and H-loop with the signature and the symmetry group characterizes 393 balls.
This is still far from 1812.
For example, among the 61 balls with signature $2P+2P_2+2P_3$ and no symmetry, 15 also have the same maximum H-run and H-loop lengths.


One could also spot specific patterns (for example, a hexagon surrounded by 6 hexagons), or even use the valve to distinguish one hexagon.
A complete and ``reasonably simple'' characterization remains an open question.
The ideal would be an analogue of the guide to recognizing wallpaper groups \cite{Rad16}, which would allow one to quickly determine the number of a ball held in one's hands, of course without reading the stamp.
An algorithm to easily distinguish, e.g., Balls No.~358 from No~439 \cite{appli}.

\section{Supplementary images {\small (Photo Tatiana \c{S}enol)}}

\begin{center}
\includegraphics[width=\textwidth]{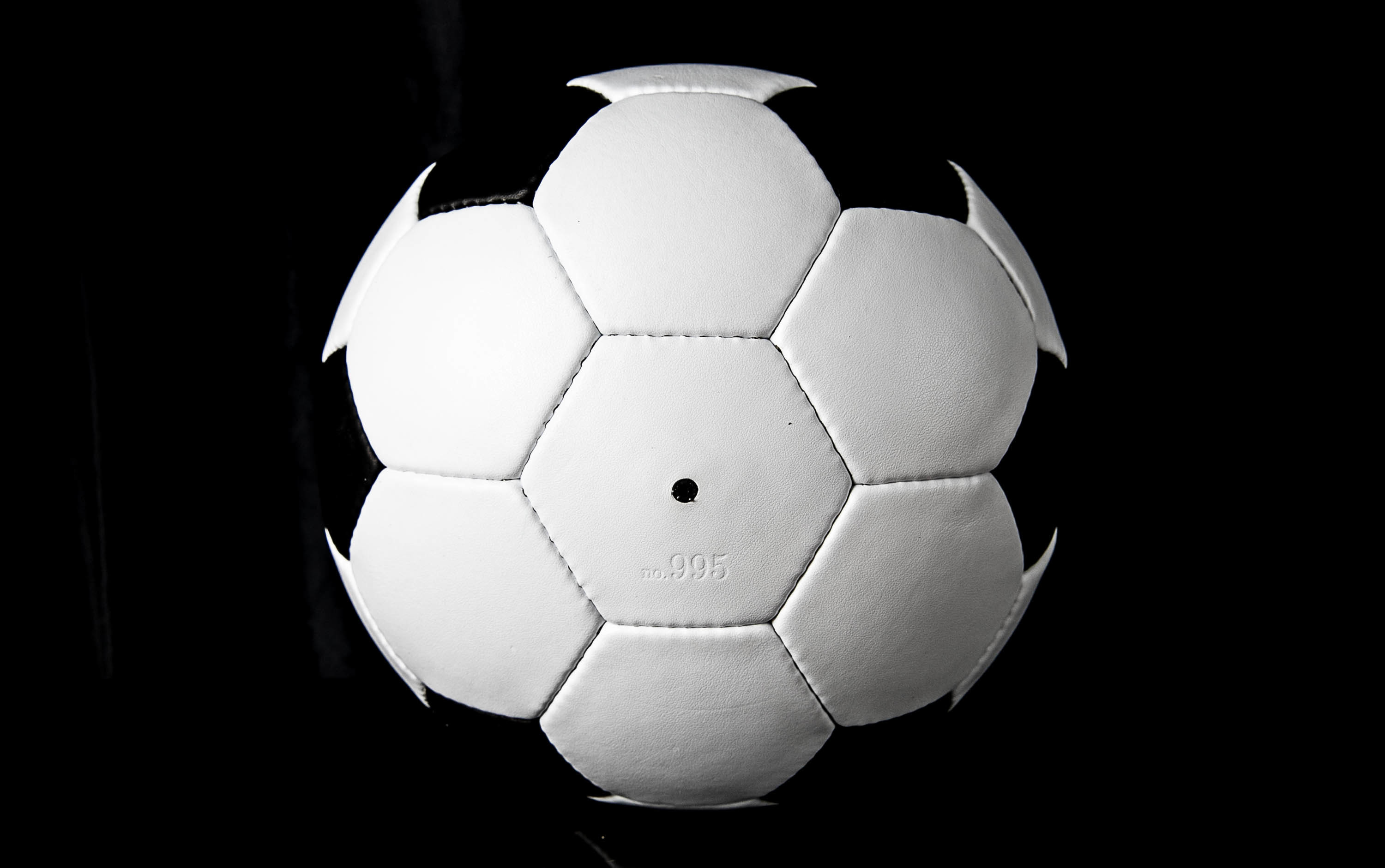}
\end{center}

\vfill

\begin{center}
\includegraphics[width=\textwidth]{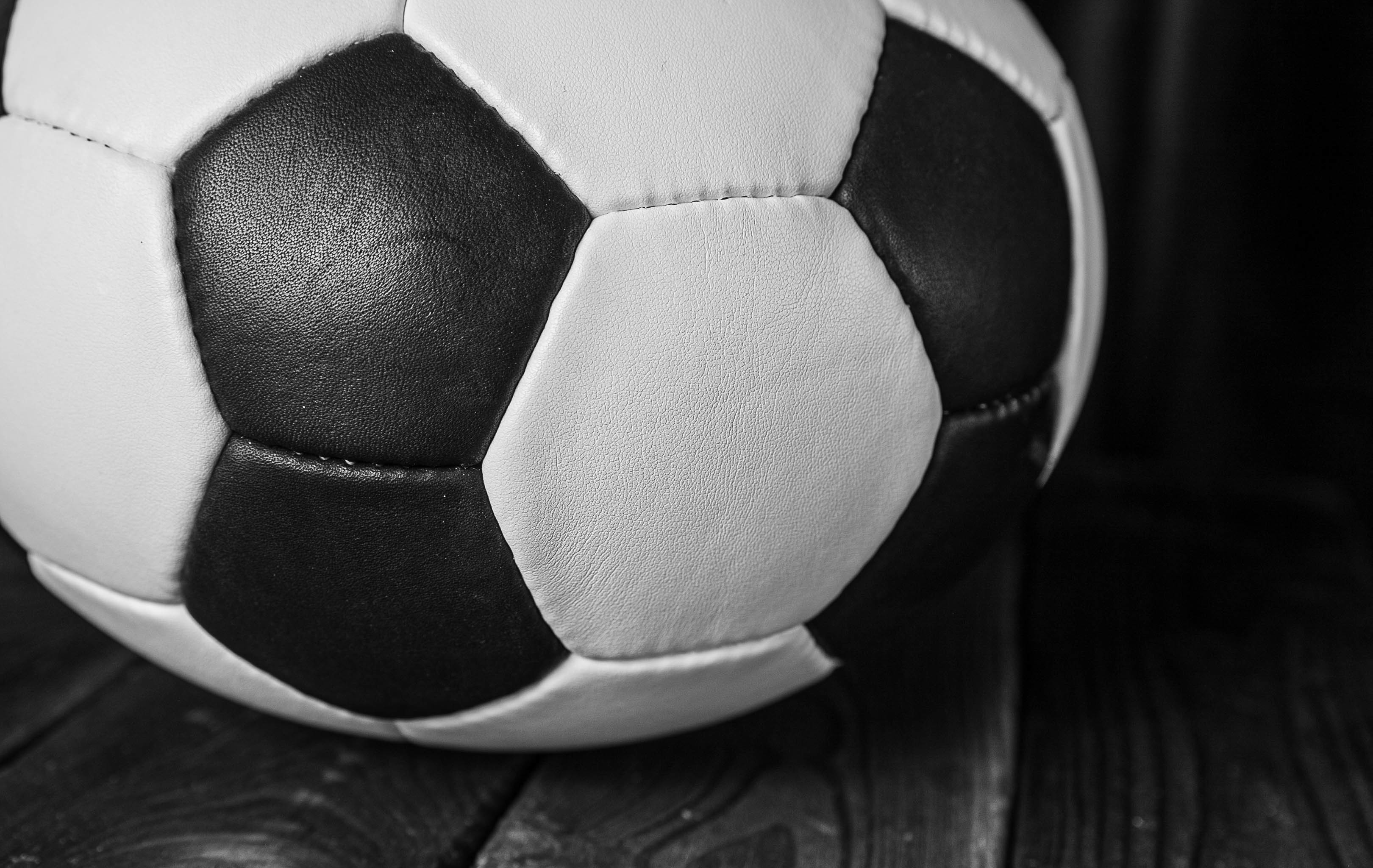}
\end{center}

\begin{center}
\includegraphics[width=\textwidth]{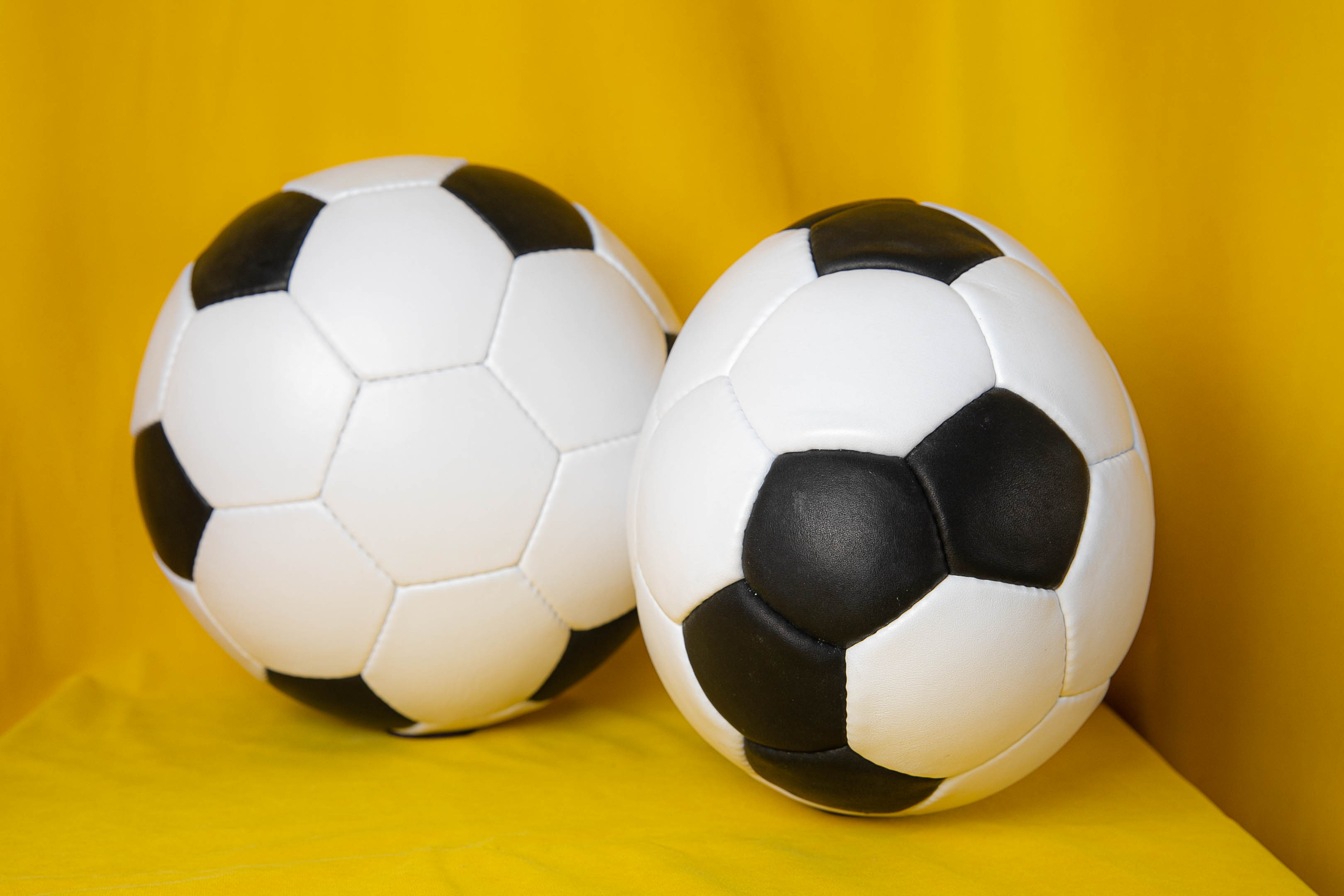}
\end{center}

\vfill

\begin{center}
\includegraphics[width=\textwidth]{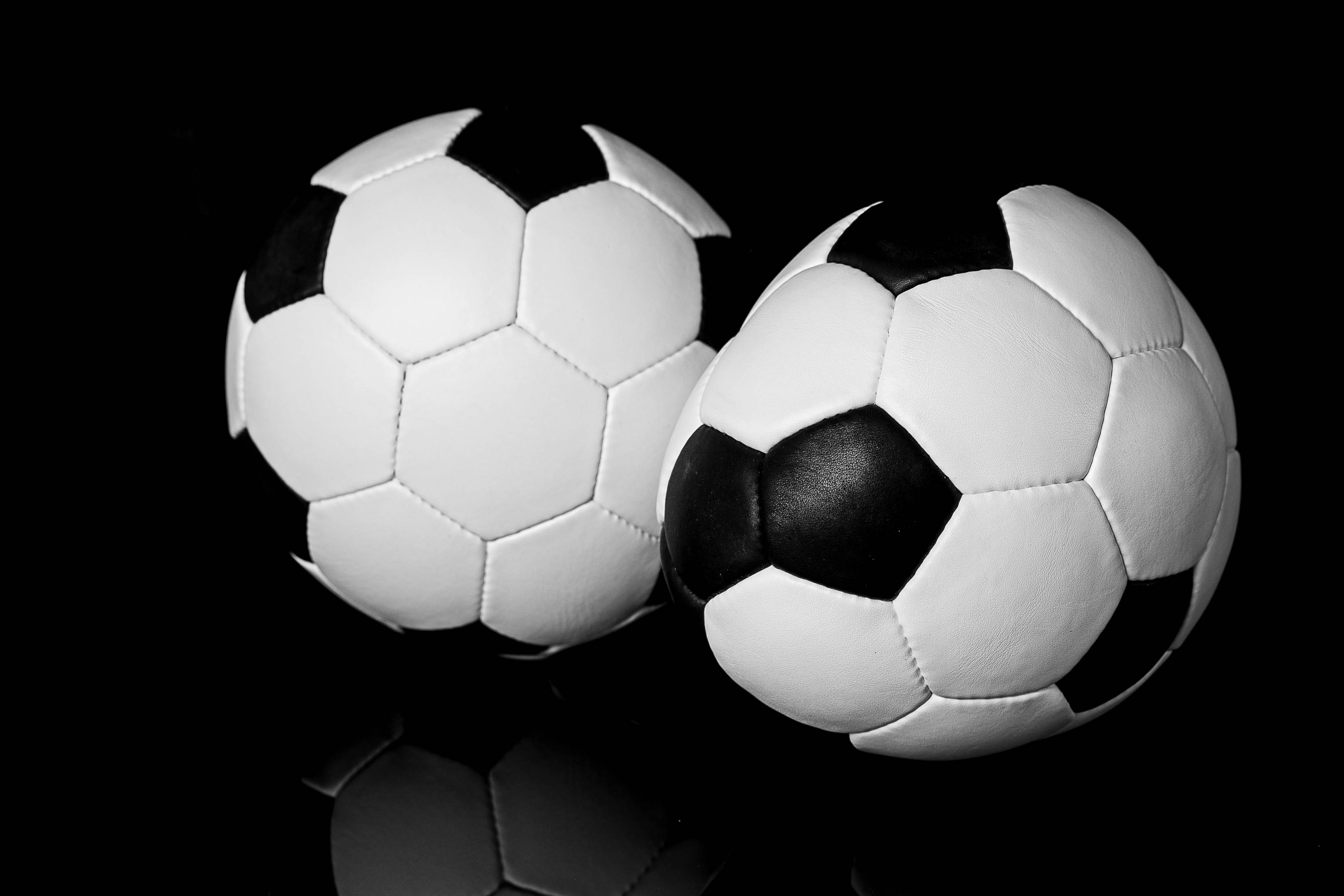}
\end{center}

\bibliographystyle{alpha}
\bibliography{telstar}

\end{document}